\documentclass[a4paper,11pt]{article}
\usepackage{aligned-overset}
\usepackage{amsmath,amsthm,amssymb}
\usepackage{authblk}
\usepackage[style=numeric-comp,maxbibnames=99,bibencoding=utf8,firstinits=true]{biblatex}
\usepackage{cancel}
\usepackage[hmargin={25mm,25mm},vmargin={30mm,35mm}]{geometry}
\usepackage{nicefrac}
\usepackage{paralist}
\usepackage[colorlinks, allcolors=blue]{hyperref}
\usepackage[compact,small]{titlesec}
\usepackage{newtxmath}
\usepackage{newtxtext}
\usepackage{booktabs}
\usepackage{siunitx}
\usepackage{tikz-cd}
\usepackage{pgfplots}
\usepackage{mathtools}
\usepackage{subfigure}

\bibliography{nsa}
\newcommand{\email}[1]{\href{mailto:#1}{#1}}

\usetikzlibrary{patterns}

\newtheorem{theorem}{Theorem}
\newtheorem{proposition}[theorem]{Proposition}
\newtheorem{lemma}[theorem]{Lemma}

\theoremstyle{remark}
\newtheorem{remark}[theorem]{Remark}
\theoremstyle{definition}

\theoremstyle{example}

\newcommand{\st}{\,:\,}
\newcommand{\Real}{\mathbb{R}}
\newcommand{\Cc}{C^\infty_{\rm c}}

\newcommand{\term}{\mathfrak{T}}
\newcommand{\res}{\mathfrak{R}}

\DeclareRobustCommand{\bvec}[1]{\boldsymbol{#1}}
\newcommand{\hbvec}[1]{\widehat{\bvec{#1}}}
\newcommand{\uvec}[1]{\underline{\bvec{#1}}}
\newcommand{\hvec}[1]{\widehat{\bvec{#1}}}
\newcommand{\huvec}[1]{\widehat{\underline{\bvec{#1}}}}

\newcommand{\bbvec}[1]{\bvec{\mathbb{#1}}}

\newcommand{\Poly}[1]{\mathcal{P}^{#1}}
\newcommand{\elements}[1]{\mathcal{T}_{#1}}
\newcommand{\faces}[1]{\mathcal{F}_{#1}}
\newcommand{\Th}{\elements{h}}
\newcommand{\TF}{\elements{F}}
\newcommand{\Fh}{\faces{h}}
\newcommand{\Fhi}{\faces{h}^{\rm i}}
\newcommand{\Fhb}{\faces{h}^{\rm b}}
\newcommand{\FT}{\faces{T}}

\newcommand{\normal}{\bvec{n}}

\DeclareMathOperator{\tr}{tr}

\newcommand{\Ih}{\uvec{I}_h}

\newcommand{\GT}{\bvec{G}_T}
\newcommand{\Gh}{\bvec{G}_h}
\newcommand{\rT}{\bvec{r}_T}
\newcommand{\DT}{D_T}
\newcommand{\Dh}{D_h}

\newcommand{\GRAD}{\bvec{\nabla}}

\newcommand{\DIV}{\GRAD\cdot}
\newcommand{\LAPL}{\bvec{\Delta}}

\newcommand{\jump}[1]{[#1]_F}
\newcommand{\average}[1]{\{#1\}_F}

\newcommand{\norm}[2]{\|#2\|_{#1}}
\newcommand{\seminorm}[2]{|#2|_{#1}}
\newcommand{\vvvert}{\vert\kern-0.25ex\vert\kern-0.25ex\vert}
\newcommand{\triplenorm}[2]{\vvvert #2\vvvert_{#1}}

\newcommand{\tF}{t_{\rm F}}

\newcommand{\upw}{\rm upw}

\newcommand{\lwr}[1]{#1_\flat}
\newcommand{\upr}[1]{#1_\sharp}

\newcommand{\Uh}{\uvec{U}_h}
\newcommand{\UT}{\uvec{U}_T}
\newcommand{\UhZ}{\uvec{U}_{h,0}}
\newcommand{\Zh}{\uvec{Z}_{h,0}}

\usepackage[normalem]{ulem}
\newcounter{corr}
\definecolor{violet}{rgb}{0.580,0.,0.827}
\newcommand{\corr}[3]{\typeout{Warning : a correction remains in page \thepage}
  \stepcounter{corr}
	      {\color{blue}\ifmmode\text{\,\sout{\ensuremath{#1}}\,}\else\sout{#1}\fi}
	      {\color{red}#2}
	      {\color{violet} #3}
}

\begin{document}

\title{A low-order hybrid method for the variable-density incompressible Navier--Stokes equations}
\author[1,2]{Mathias Dauphin}
\author[1]{Daniele A.~Di Pietro}
\author[1,3]{J\'er\^ome Droniou}
\author[1]{Alexandros Skouras}
\affil[1]{
  IMAG, Univ. Montpellier, CNRS, Montpellier, France, %
  \email{daniele.di-pietro@umontpellier.fr}, %
  \email{jerome.droniou@cnrs.fr}, %
  \email{alexandros.skouras@umontpellier.fr}
}
\affil[2]{
  Scuola Superiore Meridionale, Naples, Italy, %
  \email{mathias.dauphin-ssm@unina.it}
}
\affil[3]{School of Mathematics, Monash University, Australia}

\maketitle

%% \tableofcontents

\begin{abstract}
  In this work we introduce and analyse a new low-order method for the variable-density incompressible Navier--Stokes equations.
  The main novelty of the proposed method lies in the support of general meshes, possibly including polygonal or polyhedral elements as well as non-matching interfaces.
  We carry out a complete analysis, showing stability, existence and uniqueness of a discrete solution, and convergence of the latter to a suitably defined weak solution of the continuous problem.
  Numerical tests validate the theoretical results.
\end{abstract}

%------------------------------------------------------------------------------%

\section{Introduction}

In this work we introduce and analyse a new low-order method for the variable-density incompressible Navier--Stokes (VDINS) equations based on hybrid velocity and piecewise constant density.
The main novelty of the proposed method lies in the support of general meshes, possibly including polygonal or polyhedral elements as well as non-matching interfaces.
Stability is achieved by a careful discretisation of the advective and unsteady terms.
Specifically, a discrete counterpart of the maximum principle is obtained through an upwind discontinuous Galerkin discretisation of the density advection equation;
energy estimates for the velocity are, on the other hand, obtained by discretising the unsteady term in the momentum balance equation in the spirit of~\cite{Guermond.Quartapelle:00}.
\smallskip

The recent literature on numerical methods for partial differential equations has highlighted the relevance of supporting domain subdivisions more general than the conforming triangulations or structured grids upon which classical discretization methods are built; %
see, e.g., the introduction of~\cite{Di-Pietro.Droniou:20} for a historical perspective.
Such additional flexibility can, e.g., %
facilitate the representation of small geometric features through agglomeration~\cite{Bassi.Botti.ea:12,Antonietti.Cangiani.ea:16}, %
enable the seamless handling of internal boundaries through element cutting~\cite{Di-Pietro.Mendez.ea:25}, %
and make adaptivity through non-conforming local mesh refinement possible~\cite{Mengolini.Benedetto.ea:19}.
In the context of the VDINS equations, the latter feature is particularly interesting, as adaptive mesh refinement can be used to finely track sharp variations of the density (due to, e.g., the presence of multiple miscible phases).
\smallskip

Several works address the numerical approximation of the VDINS equations on standard meshes.
In \cite{Liu.Walkington:07}, the authors propose an inf-sup stable discontinuous Galerkin method on conforming triangulations.
A finite volume scheme relying on a staggered discretisation on quadrilateral/hexahedral meshes is proposed and analysed in~\cite{Latche.Saleh:17}; see~\cite{Gallouet.Herbin.ea:17} on the related Marker-and-Cell scheme and also the more recent contribution~\cite{Batteux.Gallouet.ea:23} addressing the case of density-dependent viscosity.
In all of the above works, the convergence analysis of the scheme under consideration is carried out using compactness techniques similar to the ones used in the present work.

A (dual) finite volume-finite element method with time splitting is proposed and numerically validated in~\cite{Calgaro.Creuse.ea:08}.
In~\cite{Goudon.Krell:14}, the authors develop a two-dimensional discrete duality finite volume method which can, in principle, support more general meshes than the above methods.
Its extension  to three dimensions, while most likely possible, does not seem entirely straightforward; see, e.g.,~\cite{Coudiere.Hubert:11} on this subject.
Numerical simulations with both the methods of~\cite{Calgaro.Creuse.ea:08} and~\cite{Goudon.Krell:14} are carried out in~\cite{Calgaro.Creuse.ea:17}.
Finally, an unconditionally stable discretisation based on the Taylor--Hood element and Temam's device for stability~\cite{Temam:79} is considered in~ \cite{Cai.Li.ea:21}, where error estimates are derived for smooth solutions and quasi-uniform triangulations.

The efficient time discretisation of the VDINS equations has also been considered in several works.
Since this topic is not in the scope of the present paper, we will only cite~\cite{Guermond.Quartapelle:00}, which introduces the idea of using the square root of the viscosity in the unsteady term of the momentum balance equation to achieve unconditional stability.
This idea is taken up in our work to ensure suitable a priori estimates on the discrete solution.
%% An artificial compressibility technique for discretisations using the pressure and momentum as main unknowns is proposed in~\cite{Cappanera.Giordano:25}, where the authors prove stability and convergence for the time semi-discrete problem under the assumption that density is advected with a maximum principle-preserving method.
%% Under similar assumptions, the stability and convergence of a semi-implicit scheme covering also density-dependent viscosity is proved in~\cite{Vu.Cappanera:25}.
\smallskip

The space semi-discrete scheme proposed and studied here is inspired by the recent literature on Hybrid High-Order (HHO) methods; see, in particular,~\cite{Di-Pietro.Krell:17,Botti.Di-Pietro.ea:19*1} concerning the constant density Navier--Stokes equations.
The discretisation of the velocity hinges on (constant) unknowns located at mesh elements and faces.
The discretisation of the viscous term in the momentum balance equation is based on the diffusive bilinear form originally introduced in \cite{Di-Pietro.Ern.ea:14};
for the convective term, we use an inherently non-dissipative formulation, analogous to the one proposed in~\cite{Di-Pietro.Krell:18};
finally, the discretisation of the unsteady term is inspired by~\cite{Guermond.Quartapelle:00} and makes use of the square root of the density variable.
The upwind discretisation of the density advection equation naturally yields a maximum principle for this variable.
Combined with the specific form of the unsteady term in the momentum balance equation, it is also a crucial ingredient to derive an energy estimate for the velocity.

A complete analysis of this method is carried out.
Existence and uniqueness of a solution to the scheme is proved using classical arguments in nonlinear ordinary differential equations.
Using a compactness argument inspired by the classical literature on finite volume and HHO methods (see, e.g., \cite{Eymard.Gallouet.ea:00,Eymard.Herbin.ea:07,Eymard.Gallouet.ea:10,Di-Pietro.Ern:10,Di-Pietro.Krell:18} and, in particular,~\cite{Latche.Saleh:17} concerning the VDINS equations), we show that the sequence of discrete solutions converges (up to a subsequence) to a weak solution of the continuous problem for general data and without assuming more regularity than required by the weak formulation.

The theoretical results are completed by numerical tests which also make it possible to investigate the orders of convergence for smooth solutions on a variety of mesh types.
\smallskip

The rest of this paper is organised as follows.
In Section~\ref{sec:continuous.problem} we present the VDINS equations and their weak formulation.
In Section~\ref{sec:discrete.setting} we introduce mesh-related notations, discrete spaces and local reconstructions.
Section~\ref{sec:discrete} contains the formulation of the semi-discrete problem, the statement of the main results of our analysis accompanied by a numerical illustration, and a discussion to bridge the scheme with the weak formulation.
A priori estimates on the discrete solution, as well as its existence and uniqueness, are proved in Section~\ref{sec:stability}.
Finally, Section~\ref{sec:convergence} contains the convergence proof.%, while numerical results are presented in Section~\ref{sec:numerical.results}.

%------------------------------------------------------------------------------%

\section{The continuous problem}\label{sec:continuous.problem}

For $d \in \{ 2, 3\}$, let $\Omega \subset \Real^d$ be an open, bounded, polygonal (if $d=2$) or polyhedral (if $d=3$) domain with boundary $\Gamma$ and outward unit normal vector $\normal_{\Omega}$, and let $\tF > 0$ be a given final time.
I what follows, given a function of time and space $\psi$, we adopt the convention that $\psi(t)$ stands for the function $\psi(t,\cdot)$ of space only.
The variable-density Navier--Stokes equations read:
Find the velocity $\bvec{u} : [0,\tF]\times\Omega \to \Real^d$,
pressure $p : (0,\tF\rbrack \times \Omega \to \Real$,
and density $\rho : [0,\tF] \times \Omega \to \Real$ such that,
for $t\in(0,\tF\rbrack$,
\begin{subequations}\label{eq:continuous.problem}
  \begin{alignat}{4}\label{eq:density}
    \partial_t \rho(t) +  \DIV(\rho\bvec{u})(t) &= 0 &\qquad& \text{in $\Omega$,}
    \\ \label{eq:momentum}
    \partial_t (\rho\bvec{u})(t)
    - \mu \LAPL \bvec{u}(t)
    + \DIV (\rho\bvec{u} \otimes \bvec{u})(t)
    + \GRAD p(t) &= \bvec{f}(t) &\qquad& \text{in $\Omega$,}
    \\ \label{eq:incompressibility}
    \DIV \bvec{u}(t) &= 0 &\qquad& \text{in $\Omega$,}
    \\ \label{eq:boundary.condition}
    \bvec{u}(t) &= \bvec{0} &\qquad& \text{on $\Gamma$,}
    \\ \label{eq:zero.average}
    \int_\Omega p(t) &= 0,
  \end{alignat}
  where $\mu > 0$ is the (constant) shear viscosity of the fluid and $\bvec{f} : \lparen 0,\tF\rbrack \times \Omega \to \Real^d$ is a body force per unit volume.
  The problem is completed by the following initial conditions:
  \begin{equation} \label{eq:initial.conditions}
    \text{
      $\bvec{u}(0) = \bvec{u}^0 $
      and
      $\rho(0) = \rho^0 $
      in $\Omega$
    }
  \end{equation}
\end{subequations}
with $\bvec{u}^0 : \Omega\to\Real^d$ such that $\DIV\bvec{u}^0 = 0$ in $\Omega$ and $\bvec{u}^0 = \bvec{0}$ on $\Gamma$,
and $\rho^0 : \Omega\to[\lwr{\rho},\upr{\rho}]$ for fixed real numbers $0<\lwr{\rho}\leq \upr{\rho}$.
For any Banach space $Y$, we let $\bvec{Y} \coloneqq Y^d$ and $\bbvec{Y}\coloneqq Y^{d\times d}$.
We additionally denote by $\bvec{Z}$ the divergence-free subspace of $\bvec{H}^1_0(\Omega)$, i.e.,
\begin{equation}\label{eq:Z}
  \bvec{Z} \coloneqq \left\{\bvec{v}\in \bvec{H}^1_0(\Omega) \st \DIV\bvec{v} = 0 \right\}.
\end{equation}
Let $\bvec{f}\in L^2(0,\tF;\bvec{L}^2(\Omega))$, %
$\bvec{u}^0\in\bvec{Z}$,
and take $\rho^0\in L^{\infty}(\Omega)$ such that
\begin{equation}\label{eq:cont.maximum.principle.density}
  \text{$0 < \lwr{\rho} \le \rho(t,\bvec{x}) \le \upr{\rho}$\:\: for almost every $(t,\bvec{x})\in(0,\tF)\times\Omega$}.
\end{equation}
A pair $(\rho,\bvec{u})$, with $\rho\in L^{\infty}(0,\tF;L^{\infty}(\Omega))$ such that $\rho\ge \lwr{\rho} > 0$ almost everywhere in $(0,t_F)\times\Omega$
and $\bvec{u}\in L^{\infty}(0,\tF;\bvec{L}^2(\Omega))\cap L^2(0,\tF;\bvec{Z})$,
is a weak solution of problem \eqref{eq:continuous.problem} if it satisfies, for all $\varphi\in \Cc(\lbrack 0,\tF\rparen\times\overline{\Omega})$
and all $\bvec{v}\in \Cc(\lbrack 0,\tF\rparen\times\Omega)^d$ such that $\DIV\bvec{v}=0$,
\begin{subequations}\label{eq:weak}
  \begin{align}\label{eq:weak:density}
    -\int_0^{\tF} \int_\Omega \rho\Big(\partial_t\varphi + \bvec{u}\cdot\GRAD\varphi \Big)
    &= \int_\Omega \rho^0\varphi(0),
    \\ \label{eq:weak:momentum}
    \int_0^{\tF}\int_\Omega \left(
    - \rho \bvec{u} \cdot \partial_t \bvec{v}
    + \mu \GRAD \bvec{u} \colon \GRAD \bvec{v}
    - \rho \bvec{u} \otimes \bvec{u} \colon \GRAD \bvec{v}
    \right)
    &= \int_\Omega \rho^0\bvec{u}^0\cdot \bvec{v}(0)
    + \int_0^{\tF}\int_\Omega \bvec{f}\cdot \bvec{v}.
  \end{align}
\end{subequations}

%------------------------------------------------------------------------------%

\section{Discrete setting}\label{sec:discrete.setting}

\subsection{Mesh and basic results}

Let $\mathcal{H} \subset \Real^+_*$ denote a countable set of mesh sizes having 0 as its unique accumulation point. We consider a sequence of refined meshes $(\Th)_{h \in \mathcal{H}}$
that is regular in the sense of~\cite[Definition~1.9]{Di-Pietro.Droniou:20}.
For all $h \in \mathcal{H}$, we denote by $\Fh = \Fhi \sqcup \Fhb$ the set of mesh faces, with $\Fhi$ and $\Fhb$ respectively collecting internal faces contained in $\Omega$ and boundary faces contained in $\Gamma$.
For all $T \in \Th$, the set $\mathcal{F}_T \coloneqq \left\{F \in \Fh \st F \subset \partial T \right\}$ collects the faces lying on the boundary of $T$ and, for all $F \in \mathcal{F}_T$, we denote by $\normal_{TF}$ the normal to $F$ pointing out of $T$. For $F\in\Fh$ we denote by $\TF$ the set of elements $T$ such that $F\in\FT$. A normal vector $\normal_F$ is associated to each internal face $F$ by fixing once and for all an (arbitrary) orientation, whereas, for boundary faces, $\normal_F$ points out of $\Omega$. For a mesh element or face $X\in\Th\cup\Fh$, we denote by $h_X$ its diameter, so that $h = \max_{T\in\Th}h_T$, and by $|X|$ its Hausdorff measure.

Throughout the paper, we often write $a \lesssim b$ to mean $a \leq Cb$ with real number $C > 0$ independent of the meshsize $h$ and, for local inequalities on a mesh element or face $X \in \Th \cup \Fh$, also on $X$.
The hidden constant can, however, possibly depend on other quantities such as the space dimension $d$, the mesh regularity parameter, the polynomial degree, the physical parameters, etc.
We also write $a \simeq b$ for ``$a \lesssim b$ and $b \lesssim a$''.
Constants are, however, named when needed in the discussion.

Let $X\in\Th\cup\Fh$, and denote by $\Poly{0}(X)$ the space of constant functions on $X$. We denote by $\pi_X : L^1(X) \rightarrow \Poly{0}(X)$ the $L^2$-orthogonal projector such that, for all $v \in L^1(X)$,
$\pi_X v \equiv \frac{1}{|X|} \int_X v$ in $X$.
The vector- and matrix-valued $L^2$-orthogonal projectors, both denoted by $\bvec{\pi}_X$, are obtained applying $\pi_X$ component-wise.
Moreover, we denote by $\pi_h$ denotes the $L^2$-orthogonal projector on $\Poly{0}(\Th)$, which satisfies $(\pi_h\phi)|_T=\pi_T\phi$ for all $T\in\Th$.
The following classical approximation properties hold (see~\cite[Appendix~A.2]{Di-Pietro.Droniou:17} and also~\cite[Chapter~1]{Di-Pietro.Droniou:20} for generalisations):
For all $p\in [1,\infty]$, all $h \in \mathcal{H}$, all $T \in \Th$, and all $q \in W^{1,p}(T)$,
\begin{equation}
  \label{eq:approximation.mesh.element}
  \norm{L^{p}(T)}{q - \pi_T q}
  + h_T^{\frac{1}{p}}\norm{L^p(\partial T)}{q - \pi_T q}
  \lesssim h_T \seminorm{W^{1,p}(T)}{q}.
\end{equation}
Moreover, by mesh regularity, $h_T |F| \lesssim |T|$ for all $F \in \FT$, and so the following (trivial) discrete trace inequality holds:
\begin{equation}\label{eq:discrete.trace.inequality}
  h_T^{\frac{1}{p}} \norm{L^p(\partial T)}{q}
  \lesssim \norm{L^p(T)}{q}
  \qquad \forall q\in\Poly{0}(T).
\end{equation}

\subsection{Discrete spaces}\label{sec:discrete.setting:discrete.spaces}

We define a hybrid space containing element-based and face-based velocities, and space of piecewise constant functions for the density:
\[
\begin{gathered}
  \Uh \coloneqq \left\{ \uvec{v}_h =\big( (\bvec{v}_T)_{T\in\Th}, (\bvec{v}_F)_{F\in\Fh} \big)\st
  \text{
    $\bvec{v}_T \in \Poly{0}(T)^d$ for all $T \in \Th$ and
    $\bvec{v}_F \in \Poly{0}(F)^d$ for all $F \in \Fh$
  }
  \right\},
  \\
  \Poly{0}(\Th) \coloneqq \left\{
  q_h \in L^2(\Omega) \st \text{
    $q_T\coloneqq q_{h|T} \in \Poly{0}(T)$ for all $T \in \Th$
  }
  \right\}.
\end{gathered}
\]
The global velocity interpolator $\Ih : \bvec{H}^1(\Omega) \rightarrow \Uh$ is such that, for all $\bvec{v} \in \bvec{H}^1(\Omega)$,
\begin{equation*}%\label{eq:global.interpolator}
  \Ih \bvec{v} \coloneqq \left( \left( \pi_T \bvec{v} \right)_{T \in \Th}, \left( \pi_F \bvec{v} \right)_{F \in \Fh} \right).
\end{equation*}
For every mesh element $T \in \Th$, we denote by $\UT$ and $\uvec{I}_T$ the restrictions to $T$ of $\Uh$ and $\Ih$, respectively. Similarly, $\uvec{v}_T = (\bvec{v}_T, (\bvec{v}_F)_{F \in \FT})$ denotes the restriction to $T$ of a generic vector $\uvec{v}_h \in \Uh$. Finally, for any $\uvec{v}_h \in \Uh$, we denote by $\bvec{v}_h$ (no underline) the broken polynomial function in $\Poly{0}(\Th)^d$ such that $(\bvec{v}_h)_{|T} = \bvec{v}_T$ for all $T \in \Th$.

We define on $\Uh$ the following seminorm:
\begin{equation}\label{eq:norm.1.h}
  \text{%
    $\norm{1,h}{\uvec{v}_h}^2 \coloneqq \sum_{T \in \Th} \norm{1,T}{\uvec{v}_T}^2$
    with
    $\norm{1,T}{\uvec{v}_T}^2 \coloneqq h_T^{-1} \sum_{F \in \FT} \norm{\bvec{L}^2(F)}{\bvec{v}_F - \bvec{v}_T}^2$
    for all $T \in \Th$.
  }
\end{equation}
The discrete velocity and density are functions of time taking values in the following spaces, respectively:
\[
\UhZ \coloneqq \left\{ \uvec{v}_h \in \Uh \st
\text{
  $\bvec{v}_F = \bvec{0}$ for all $F \in \Fhb$
}\right\},\qquad
\Phi_h \coloneqq \Poly{0}(\Th).
\]
The map $\norm{1,h}{{\cdot}}$ defines a norm on $\UhZ$, as can be proved using the following result.
\begin{proposition}[Discrete Sobolev inequalities]
  For all
  $p \in \lbrack 1, \infty \rparen$ if $d = 2$
  and all $p \in [1, 6]$ if $d = 3$,
  it holds
  \begin{equation}\label{eq:discrete.sobolev}
    \left(
    \norm{\bvec{L}^p(\Omega)}{\bvec{v}_h}^p
    + \sum_{T \in \Th} h_T \norm{\bvec{L}^p(\partial T)}{\bvec{v}_T}^p
    + \sum_{T \in \Th} \sum_{F \in \FT} h_T \norm{\bvec{L}^p(F)}{\bvec{v}_F}^p
    \right)^{\frac1p}
    \lesssim \norm{1,h}{\uvec{v}_h}
    \qquad \forall \uvec{v}_h \in \UhZ.
  \end{equation}
\end{proposition}

\begin{proof}
  Let $\uvec{v}_h \in \UhZ$ and denote by $\term_1,\term_2,\term_3$ the terms inside parentheses in the left-hand side.
  Using the trace inequality \eqref{eq:discrete.trace.inequality} we have $\term_2\lesssim \sum_{T \in \Th} \norm{\bvec{L}^p(T)}{\bvec{v}_T}^p  = \term_1$ so, by~\cite[Theorem~6.40]{Di-Pietro.Droniou:20},
  \begin{equation}\label{eq:discrete.sobolev:term.1}
    \term_1+\term_2\lesssim\term_1 \lesssim \norm{1,h}{\uvec{v}_h}^p.
  \end{equation}
  To estimate the third term, we first note that we can assume $p\ge 2$. Indeed, if $p<2$, H\"older inequalities (on the Lebesgue norm on $F$ and the sums) with exponents $(\frac2p,\frac{2}{2-p})$ give
  \[
  \term_3 \le \left(\sum_{T \in \Th} \sum_{F \in \FT} h_T \norm{\bvec{L}^2(F)}{\bvec{v}_F}^2\right)^{\frac{p}{2}}
  \left(\sum_{T \in \Th} \sum_{F \in \FT} h_T |F|\right)^{\frac{2-p}{p}}\lesssim
  \left(\sum_{T \in \Th} \sum_{F \in \FT} h_T \norm{\bvec{L}^2(F)}{\bvec{v}_F}^2\right)^{\frac{p}{2}},
  \]
  where the conclusion follows from $\sum_{F\in\FT}h_T |F|\simeq |T|$ by mesh regularity. Hence, the estimate of $\term_3$ for $p<2$ follows from the same estimate for $p=2$.
  
  We therefore take $p\ge 2$ and write, inserting $\pm \bvec{v}_T$ into the norm and using a triangle inequality followed by the fact that $(a+b)^p \lesssim (a^p + b^p)$ (cf., e.g.,~\cite[Eq.~(6.4)]{Botti.Castanon-Quiroz.ea:21}),
  \begin{equation}\label{eq:discrete.sobolev:intermediate}
    \term_3
    \lesssim
    \sum_{T \in \Th} \sum_{F \in \FT} h_T \norm{\bvec{L}^p(F)}{\bvec{v}_F - \bvec{v}_T}^p
    + \sum_{T \in \Th} h_T \norm{\bvec{L}^p(\partial T)}{\bvec{v}_T}^p
    =
    \sum_{T \in \Th} \sum_{F \in \FT} h_T \norm{\bvec{L}^p(F)}{\bvec{v}_F - \bvec{v}_T}^p
    + \term_2.
  \end{equation}
  Since $\bvec{v}_F - \bvec{v}_T$ is constant, we have 
  \[
  h_T \norm{\bvec{L}^p(F)}{\bvec{v}_F - \bvec{v}_T}^p=h_T|F|\, |\bvec{v}_F - \bvec{v}_T|^p=\left(h_T^{-1}\norm{\bvec{L}^2(F)}{\bvec{v}_F-\bvec{v}_T}^2\right)^{\frac{p}{2}} h_T^{1+\frac{p}{2}}|F|^{1-\frac{p}{2}}.
  \]
  By mesh regularity, $|F|\simeq h_T^{d-1}$ from which we deduce
  \begin{equation}\label{eq:discrete.sobolev:intermediate.0}
    h_T \norm{\bvec{L}^p(F)}{\bvec{v}_F - \bvec{v}_T}^p\lesssim \left(h_T^{-1}\norm{\bvec{L}^2(F)}{\bvec{v}_F-\bvec{v}_T}^2\right)^{\frac{p}{2}} h_T^{d+\frac{p}{2}(d-1)}\lesssim \left(h_T^{-1}\norm{\bvec{L}^2(F)}{\bvec{v}_F-\bvec{v}_T}^2\right)^{\frac{p}{2}},
  \end{equation}
  where the conclusion comes from $d+\frac{p}{2}(d-1)\ge 0$ (by choice of $p$) and $h_T\le\operatorname{diam}(\Omega)\lesssim 1$. Since $\frac{p}{2} \ge 1$, we trivially have
  \[
  \sum_{T \in \Th} \sum_{F \in \FT} \left(h_T^{-1}\norm{\bvec{L}^2(F)}{\bvec{v}_F-\bvec{v}_T}^2\right)^{\frac{p}{2}}
  \le \left(\sum_{T \in \Th} \sum_{F \in \FT} h_T^{-1}\norm{\bvec{L}^2(F)}{\bvec{v}_F-\bvec{v}_T}^2\right)^{\frac{p}{2}}
  \overset{\eqref{eq:norm.1.h}}\le \norm{1,h}{\uvec{v}_h}^p.
  \]
  Combined with \eqref{eq:discrete.sobolev:intermediate.0} and \eqref{eq:discrete.sobolev:intermediate}, this shows that $\term_3\lesssim \norm{1,h}{\uvec{v}_h}^p+\term_2$.
  The proof is completed by recalling~\eqref{eq:discrete.sobolev:term.1}.
\end{proof}

\subsection{Discrete differential operators}

Let $T \in \Th$.
We define the usual HHO gradient reconstruction operator $\GT : \UT \to \Poly{0}(T)^{d\times d}$ such that, for all $\uvec{v}_T \in \UT$,
\begin{equation}\label{eq:GT}
  \int_T \GT \uvec{v}_T : \bvec{\tau}
  \coloneqq \sum_{F \in \FT} \int_F (\bvec{v}_F - \bvec{v}_T) \cdot \bvec{\tau}\normal_{TF}
  \qquad \forall \bvec{\tau} \in \Poly{0}(T)^{d\times d}.
\end{equation}
We also define the discrete divergence operator $\DT : \UT \to \Poly{0}(T)$ such that, for all $\uvec{v}_T \in \UT$, $\DT\uvec{v}_T = \tr (\GT \uvec{v}_T)$, i.e.,
\begin{equation}\label{eq:DT}
  \int_T \DT\uvec{v}_T
  = \sum_{F \in \FT}\int_F \bvec{v}_F \cdot \normal_{TF}.
\end{equation}
The global versions of the gradient and divergence operators are $\Gh : \Uh \to \Poly{0}(\Th)^{d\times d}$ and $\Dh : \Uh \to \Poly{0}(\Th)$ such that, for all $\uvec{v}_h \in \Uh$,
\begin{equation}\label{eq:Gh.Dh}
  \text{
    $(\Gh \uvec{v}_h)_{|T} \coloneqq \GT \uvec{v}_T$
    and
    $(\Dh \uvec{v}_h)_{|T} \coloneqq \DT \uvec{v}_T$
    for all $T \in \Th$.
  }
\end{equation}
The discrete counterpart of the subspace of divergence-free functions $\bvec{Z}$ (see~\eqref{eq:Z}) is then defined setting
\begin{equation}\label{eq:Zh}
  \Zh \coloneqq \left\{ \uvec{v}_h \in \UhZ \st
  \Dh\uvec{v}_h = 0 \right\}.
\end{equation}

\subsection{Discrete velocity $L^2$-product and norm}

To induce an $\bvec{L}^2$-structure on $\UhZ$, we define the inner product $(\cdot,\cdot)_{0,h}: \UhZ\times \UhZ \to \Real$ such that, for all $(\uvec{w}_h,\uvec{v}_h)\in \UhZ\times \UhZ$,
\begin{equation}\label{eq:L2.discrete.product}
  \text{
    $(\uvec{w}_h,\uvec{v}_h)_{0,h} \coloneqq \sum_{T \in \Th}(\uvec{w}_T,\uvec{v}_T)_{0,T}$
    with
    $(\uvec{w}_T,\uvec{v}_T)_{0,T} \coloneqq \int_T \bvec{w}_T\cdot\bvec{v}_T + j_{T}(\uvec{w}_T,\uvec{v}_T)$,
  }
\end{equation}
where
\[
j_T(\uvec{w}_T,\uvec{v}_T)\coloneqq h_T \sum_{F \in \FT\cap\Fhi} \int_F (\bvec{w}_F - \bvec{w}_T)\cdot (\bvec{v}_F - \bvec{v}_T).
\]
For future use, we also define the global stabilisation bilinear form $j_h : \UhZ \times \UhZ \to \Real$ such that, for all $(\uvec{w}_h,\uvec{v}_h) \in \Uh \times \Uh$,
\[
j_h(\uvec{w}_h,\uvec{v}_h) = \sum_{T \in \Th} j_T(\uvec{w}_T,\uvec{v}_T)
\]
and we note that it satisfies the following consistency property \cite[Proposition 2.2]{Di-Pietro.Droniou:20}:
For all $\bvec{v}\in\bvec{H}^1_0(\Omega)$,
\begin{equation}\label{eq:jh.assumption}
  j_h(\Ih \bvec{v}, \Ih \bvec{v})^{\frac12}
  \lesssim h \seminorm{\bvec{H}^1(\Omega)}{\bvec{v}}.
\end{equation}
Global and local $L^2$-like norms are respectively given,
for all $\uvec{v}_h \in \UhZ$, by
\begin{equation}\label{eq:L2.discrete.norm}
  \norm{0,h}{\uvec{v}_h}
  \coloneqq (\uvec{v}_h,\uvec{v}_h)^{\frac12}_{0,h}, \quad \text{ and } \quad \norm{0,T}{\uvec{v}_T}\coloneqq (\uvec{v}_T,\uvec{v}_T)^{\frac12}_{0,T}, \quad \text{ for all } T \in \Th.
  %  \text{%
  %    $\norm{0,h}{\uvec{v}_h}
  %    \coloneqq (\uvec{v}_h,\uvec{v}_h)^{\frac12}_{0,h}$
  %    and
  %    $\norm{0,T}{\uvec{v}_T}\coloneqq (\uvec{v}_T,\uvec{v}_T)^{\frac12}_{0,T}$
  %    for all $T\in\Th$.
  %  }
\end{equation}
From \eqref{eq:jh.assumption} and the standard Poincar\'e inequality in $\bvec{H}^1_0(\Omega)$, the $\norm{0,h}{{\cdot}}$-continuity property hereafter for the global interpolator $\Ih$ follows:
\begin{equation}\label{eq:Ih:continuity.0.h}
  \norm{0,h}{\Ih \bvec{v}} \lesssim \seminorm{\bvec{H}^1(\Omega)}{\bvec{v}}.
\end{equation}

%------------------------------------------------------------------------------%

\section{Space semi-discrete problem}\label{sec:discrete}

\subsection{Viscous term}

The discretisation of the viscous term is based on the bilinear form $a_h : \UhZ \times \UhZ \to \Real$ such that, for all $(\uvec{w}_h, \uvec{v}_h) \in \UhZ \times \UhZ$,
\[
\text{
  $a_h(\uvec{w}_h, \uvec{v}_h) \coloneqq \sum_{T \in \Th} a_T(\uvec{w}_T, \uvec{v}_T)$
  with
  $a_T(\uvec{w}_T, \uvec{v}_T) \coloneqq \int_T \GT \uvec{w}_T : \GT \uvec{v}_T
  + s_T(\uvec{w}_T,\uvec{v}_T)$,
}
\]
where, denoting by $\rT \uvec{v}_T \in \Poly{1}(T)^d$ the unique polynomial such that
$\rT\uvec{v}_T(\bvec{x})\coloneqq \bvec{v}_T + \GT\uvec{v}_T\cdot(\bvec{x}-\overline{\bvec{x}}_T)$ with $\overline{\bvec{x}}_T\coloneqq \frac{1}{\seminorm{}{T}}\int_T \bvec{x}$ centre of mass of $T$,
\[
s_T(\uvec{w}_T,\uvec{v}_T)
\coloneqq
\frac{1}{h_T} \sum_{F \in \FT} \int_F \bvec{\pi}_F (\rT \uvec{w}_T - \bvec{w}_F) \cdot \bvec{\pi}_F (\rT \uvec{v}_T - \bvec{v}_F).
\]
The symmetric positive-definite bilinear form $a_h$ induces on $\UhZ$ the $H^1$-like norm  $\norm{a,h}{{\cdot}} : \UhZ \to \Real$ such that, for all $\uvec{v}_h \in \UhZ$,
\begin{equation}\label{eq:norm.a.h}
  \norm{a,h}{\uvec{v}_h} \coloneqq a_h(\uvec{v}_h, \uvec{v}_h)^{\frac12} \simeq \norm{1,h}{\uvec{v}_h}
\end{equation}
(see~\cite[Chapter~2]{Di-Pietro.Droniou:20} concerning the equivalence with the $\norm{1,h}{{\cdot}}$-norm). In addition, we define the following space-time discrete norm:
\begin{equation}\label{eq:space-time.norm}
  \norm{L^2(0,\tF;\,\UhZ)}{\uvec{v}_h}\coloneq\left(\int_0^{\tF} \norm{a,h}{\uvec{v}_h(t)}^2 \,dt\right)^{\frac12}
  \qquad
  \forall \uvec{v}_h\in L^2(0,\tF;\,\UhZ).
\end{equation}
For future use, we note the following compactness result, the proof of which is a straightforward adaptation of~\cite[Proposition~5.6]{Di-Pietro.Droniou:17} (see~\cite[Theorem~6.8]{Di-Pietro.Droniou:20} for an adaptation to Neumann boundary conditions).
\begin{proposition}[Properties of $\Gh$]\label{prop:properties.Gh}
  Let $(\uvec{v}_h)_{h\in\mathcal{H}}$ denote a sequence such that $\uvec{v}_h\in L^2(0,\tF;\,\UhZ)$ for all $h\in\mathcal{H}$, and $(\norm{L^2(0,\tF;\,\UhZ)}{\uvec{v}_h})_{h\in\mathcal{H}}$ is bounded.
  Then, there is $\bvec{v}\in L^2(0,\tF;\bvec{H}^1_0(\Omega))$ such that, up to a subsequence, as $h\to 0$,
  \begin{enumerate}
  \item $\bvec{v}_h \rightharpoonup \bvec{v}$ weakly in $L^2(0,\tF;\bvec{L}^q(\Omega))$ for all $1 \leq q < q^*$, where $q^*$ is such that
    \begin{equation}\label{eq:q*}
      q^* \coloneqq
      \begin{cases}
	\infty &\text{if $d = 2$},\\
	6 &\text{if $d = 3$};
      \end{cases}
    \end{equation}
  \item $\Gh \uvec{v}_h \rightharpoonup \GRAD \bvec{v}$ weakly in $L^2(0,\tF;\bbvec{L}^2(\Omega))$.
  \end{enumerate}
\end{proposition}

\begin{proposition}[Sequential compactness of $a_h$]\label{prop:properties.ah}
  Let $(\uvec{w}_h)_{h\in\mathcal{H}}$ denote a sequence such that $\uvec{w}_h\in L^2(0,\tF;\,\UhZ)$ for each $h\in\mathcal{H}$, and $(\norm{L^2(0,\tF;\,\UhZ)}{\uvec{w}_h})_{h\in\mathcal{H}}$ is bounded.
  Then, there exists $\bvec{w}\in L^2(0,\tF;\bvec{H}^1_0(\Omega))$ such that, for all $\bvec{v}\in \Cc( [0,\tF)\times\Omega)^d$, up to a subsequence as $h\rightarrow 0$,
    \begin{equation}\label{eq:sequential.compactness.viscous.term}
      \int_0^{\tF}a_h(\uvec{w}_h, \Ih\bvec{v})
      \rightarrow \int_0^{\tF} \int_\Omega \GRAD \bvec{w} \colon \GRAD \bvec{v}.
    \end{equation}
\end{proposition}

\begin{proof}
  Straightforward adaptation of~\cite[Proposition~2]{Di-Pietro.Krell:18} based on Proposition~\ref{prop:properties.Gh}.
\end{proof}

\subsection{Convective term}

The convective term in the momentum balance equation is discretised through the trilinear form $c_h : \Uh^3 \to \Real$ such that, for all $(\uvec{w}_h, \uvec{v}_h, \uvec{z}_h) \in \Uh^3$,
\begin{equation}\label{eq:ch}
  c_h(\uvec{w}_h, \uvec{v}_h, \uvec{z}_h)
  \coloneq
  \frac12 \sum_{T \in \Th} \sum_{F \in \FT} \int_F (\bvec{w}_F \cdot \normal_{TF})\left(
  \bvec{v}_F \cdot \bvec{z}_T - \bvec{v}_T \cdot \bvec{z}_F
  \right).
\end{equation}
The following non-dissipativity property is of immediate verification:
For all $\uvec{w}_h,\uvec{v}_h \in \Uh$,
\begin{equation}\label{eq:ch:non.dissipativity}
  c_h(\uvec{w}_h, \uvec{v}_h, \uvec{v}_h) = 0.
\end{equation}
In the present setting, the first argument of this trilinear form will invariably be obtained from a discrete velocity field $\uvec{w}_h \in \UhZ$ and a discrete density field $\eta_h \in \Phi_h$ setting
\[
\underline{(\eta\bvec{w})}_h
\coloneqq \big(
(\eta_T\bvec{w}_T)_{T \in \Th}, (\eta_F(\uvec{w}_h)\bvec{w}_F)_{F \in \Fh}
\big) \in \UhZ
\]
where $(\eta_F(\uvec{w}_h))_{F \in \Fhi}$ is the family of upwind values with respect to the face velocities $(\bvec{w}_F)_{F \in \Fhi}$ such that, for all $F \in \Fhi$,
\begin{equation}\label{eq:eta.F}
  \text{%
    $\eta_F(\uvec{w}_h) \coloneqq \eta_T$
    with $T\in\TF$ such that $\bvec{w}_F\cdot\normal_{TF}\ge 0$.
  }
\end{equation}
Equivalently, denoting by $s^\pm = \frac{|s| \pm s}{2}$ the positive and negative parts of $s\in\Real$ and setting $\TF=\{T,T'\}$,
\begin{equation}\label{eq:fv.upwind.flux}
  \eta_{F}(\uvec{w}_h) \bvec{w}_F \cdot \normal_{TF} = (\bvec{w}_F\cdot\normal_{TF})^+\eta_T - (\bvec{w}_F\cdot\normal_{TF})^-\eta_{T'}.
\end{equation}
Most of the times, the dependence of the face density on the velocity field is obvious from the context.
In such circumstances, we will simply write $\eta_F$ instead of $\eta_F(\uvec{w}_h)$.

\begin{proposition}[Boundedness of $c_h$]\label{prop:ch:boundedness}
  For all $(\uvec{v}_h,\uvec{w}_h,\uvec{z}_h) \in \UhZ^3$ and all $\eta_h\in \Phi_h$, it holds
  \begin{equation}\label{eq:ch:boundedness}
    c_h(\underline{(\eta\bvec{w})}_h,\uvec{v}_h,\uvec{z}_h)
    \lesssim \norm{L^\infty(\Omega)}{\eta_h} \norm{a,h}{\uvec{w}_h}\norm{a,h}{\uvec{v}_h}\norm{a,h}{\uvec{z}_h}.
  \end{equation}
\end{proposition}

\begin{proof}
  By definition of the trilinear form \eqref{eq:ch}, adding and subtracting $\bvec{v}_T \cdot \bvec{z}_T$ to the second factor of the integrand and using H\"older inequalities along with $\norm{\bvec{L}^\infty(F)}{\normal_{TF}} \le 1$, we can write
  \[
  \begin{aligned}
    &c_h(\underline{(\eta\bvec{w})}_h,\uvec{v}_h,\uvec{z}_h)
    \\
    &\quad\begin{aligned}[t]
    &\lesssim
    \norm{L^\infty(\Omega)}{\eta}
    \sum_{T \in \Th} \sum_{F \in \FT} \norm{\bvec{L}^4(F)}{\bvec{w}_F}
    \left(
    \norm{\bvec{L}^2(F)}{\bvec{v}_F- \bvec{v}_T} \norm{\bvec{L}^4(F)}{\bvec{z}_T}
    + \norm{\bvec{L}^4(F)}{\bvec{v}_T} \norm{\bvec{L}^2(F)}{\bvec{z}_T-\bvec{z}_F}
    \right)
    \\
    &\lesssim
    \norm{L^\infty(\Omega)}{\eta} \left(
    \sum_{T \in \Th} \sum_{F \in \FT} h_T \norm{\bvec{L}^4(F)}{\bvec{w}_F}^4
    \right)^{\frac14}
    \norm{1,h}{\uvec{v}_h}
    \left(
    \sum_{T \in \Th} h_T \norm{\bvec{L}^4(\partial T)}{\bvec{z}_T}^4
    \right)^{\frac14}
    \\
    &\quad
    + \norm{L^\infty(\Omega)}{\eta} \left(
    \sum_{T \in \Th} \sum_{F \in \FT} h_T \norm{\bvec{L}^4(F)}{\bvec{w}_F}^4
    \right)^{\frac14}
    \left(
    \sum_{T \in \Th} h_T \norm{\bvec{L}^4(\partial T)}{\bvec{v}_T}^4
    \right)^{\frac14}
    \norm{1,h}{\uvec{z}_h},
    \end{aligned}
  \end{aligned}
  \]
  where, in the second step, we have used H\"older inequalities on the sums after writing $1 = h_T^{\frac14} ~ h_T^{\frac14} ~ h_T^{-\frac12}$.
  We next use~\eqref{eq:discrete.sobolev} with $p = 4$ to estimate the terms in parenthesis and obtain $c_h(\underline{(\eta\bvec{w})}_h,\uvec{v}_h,\uvec{z}_h) \lesssim \norm{L^\infty(\Omega)}{\eta_h} \norm{1,h}{\uvec{w}_h}\norm{1,h}{\uvec{v}_h}\norm{1,h}{\uvec{z}_h}$.
  The norm equivalence in~\eqref{eq:norm.a.h} then gives~\eqref{eq:ch:boundedness}.
\end{proof}

\subsection{Density advection term}

For the advection of the density, we use an upwind finite volume method based on the function $d_h : \Zh \times \Phi_h^2 \to \Real$ such that, for all $(\uvec{w}_h, \eta_h, \chi_h) \in  \Zh \times \Phi_h^2$,
\begin{equation}\label{eq:dh}
  d_h(\uvec{w}_h, \eta_h, \chi_h)
  \coloneq
  \sum_{T \in \Th} \sum_{F \in \FT} \int_F \eta_{F}(\uvec{w}_h) \, (\bvec{w}_F \cdot \normal_{TF}) \, \chi_T,
\end{equation}
where $\eta_F(\uvec{w}_h)$ is defined by~\eqref{eq:eta.F}.
A reformulation of $d_h$ better suited to pass to the limit is
\begin{equation}\label{eq:dh.bis}
  d_h(\uvec{w}_h, \eta_h, \chi_h)
  = - \sum_{F \in \Fh} \int_F (\bvec{w}_F \cdot \normal_F) \jump{\eta_h} \average{\chi_h}
  + \sum_{F \in \Fh} \int_F \frac{|\bvec{w}_F \cdot 	\normal_F|}{2} \jump{\eta_h}\,\jump{\chi_h}.
\end{equation}
Above, boundary face contributions are zero since $\bvec{w}_F$ vanishes whenever $F\in\Fhb$ while, for all $F\in\Fhi$, ordering the two elements $T,T'\in\TF$ so that $\normal_{TF}=\normal_F$, the average and jump operators are respectively given by
\[
\text{%
  $\average{\zeta_h}\coloneq\frac{(\zeta_{T})_{|F}+(\zeta_{T'})_{|F}}{2}$
  and $\jump{\zeta_h}\coloneq(\zeta_{T'})_{|F}-(\zeta_{T})_{|F}$
  for all $\zeta_h \in \Phi_h$.
}
\]
To pass from~\eqref{eq:dh} to~\eqref{eq:dh.bis}, a minor adaptation of classical arguments (see, e.g.,~\cite[Lemma~2.20]{Di-Pietro.Ern:12}) is required to account for the use of face velocities satisfying $\Dh \uvec{w}_h = 0$ instead of a divergence-free velocity field defined over the entire domain $\Omega$.
Similar adaptations also yield the following partial coercivity result:
\begin{equation}\label{eq:dh:partial coercivity}
  d_h(\uvec{w}_h, \eta_h, \eta_h)
  = \frac12\sum_{F \in \Fh} \int_F |\bvec{w}_F \cdot \normal_F| \jump{\eta_h}^2
  \eqcolon
  \seminorm{\bvec{w},\upw,h}{\eta_h}^2
  \qquad
  \forall (\uvec{w}_h, \eta_h) \in \Zh \times \Phi_h.
\end{equation}
The following lemma suggests that $d_h$ is a consistent approximation of the second term in the left-hand side of the weak density advection equation~\eqref{eq:weak:density}.
\begin{lemma}[Consistency of $d_h$]\label{lem:dh:consistency}
  For all $\rho_h \in \Phi_h$ such that $\rho_h \le \upr{\rho}$ almost everywhere in $\Omega$,
  all $\uvec{u}_h \in \Zh$,
  and all $\varphi \in \Cc(\Omega)$, it holds
  \begin{equation}\label{eq:dh:consistency}
    d_h(\uvec{u}_h,\rho_h,\pi_h\varphi)
    = -\int_{\Omega} \rho_h \bvec{u}_h \cdot \GRAD \varphi + \res,
  \end{equation}
  where
  \begin{equation}\label{eq:dh:consistency:residual}
    |\res| \lesssim 
    h^{\frac12} \seminorm{\bvec{u},\upw,h}{\rho_h} \norm{a,h}{\uvec{u}_h}^{\frac12} \seminorm{W^{1,4}(\Omega)}{\varphi}
    + h \upr{\rho} \norm{a,h}{\uvec{u}_h} \seminorm{H^1(\Omega)}{\varphi}.
  \end{equation}
\end{lemma}

\begin{proof}
  We follow the ideas of \cite[Proposition 8.5]{Latche.Saleh:17}.
  Recalling the reformulation~\eqref{eq:dh.bis} of $d_h$, we write
  \[
  d_h(\uvec{u}_h,\rho_h,\pi_h\varphi)
  = - \sum_{F \in \Fh} \int_F (\bvec{u}_F\cdot\normal_F)\jump{\rho_h}\average{\pi_h\varphi}
  + \sum_{F \in \Fh} \int_F \frac{|\bvec{u}_F \cdot 	\normal_F|}{2} \jump{\rho_h}\,\jump{\pi_h\varphi}.
  \]
  Noticing that $(\bvec{u}_F \cdot \normal_F) \jump{\rho_h} = \sum_{T \in \TF} (\bvec{u}_F \cdot \normal_{TF}) \rho_T$,
  writing $\sum_{F \in \Fh} \sum_{T \in \TF} \bullet = \sum_{T \in \Th} \sum_{F \in \FT} \bullet$,
  and denoting by $\res_1$ the second term in the right-hand side of the above expression,
  we go on writing
  \[
  \begin{aligned}
    d_h(\uvec{u}_h,\rho_h,\pi_h\varphi)
    & = -\sum_{T \in \Th} \sum_{F \in \FT} \int_F \rho_T (\bvec{u}_F\cdot\normal_{TF})\average{\pi_h\varphi}
    + \res_1
    \\
    & = -\sum_{T \in \Th} \sum_{F \in \FT} \int_F \rho_T (\bvec{u}_F\cdot\normal_{TF}) \varphi
    + \res_1 + \underbrace{
      \sum_{T \in \Th} \sum_{F \in \FT} \int_F \rho_T (\bvec{u}_F\cdot\normal_{TF}) \left( \varphi - \average{\pi_h\varphi}\right)
    }_{\res_2}
    \\
    & = -\sum_{T \in \Th} \int_{\partial T} \rho_T (\bvec{u}_T \cdot \normal_{T}) \varphi
    + \res_1 + \res_2
    + \underbrace{
      \sum_{T \in \Th} \sum_{F \in \FT}\int_F \rho_T (\bvec{u}_T - \bvec{u}_F)\cdot\normal_{TF} \varphi
    }_{\res_3}
    \\
    & = -\int_\Omega \rho_h \bvec{u}_h \cdot \GRAD \varphi
    + \res_1 + \res_2 + \res_3,
  \end{aligned}
  \]
  where $\normal_{T}$ is the outer normal to $\partial T$ and the conclusion follows from an element-by-element integration by parts after noticing that $\DIV(\rho_T \bvec{u}_T) = 0$.
  For the first residual term, we have
  \[
  \begin{aligned}
    \res_1
    &\lesssim \left(
    \sum_{F \in \Fh}\int_F \seminorm{}{\bvec{u}_F\cdot\normal_F}\jump{\rho_h}^2
    \right)^{\frac12} \left(
    \sum_{F \in \Fh}\int_F \seminorm{}{\bvec{u}_F\cdot\normal_F}\jump{\pi_h\varphi}^2
    \right)^{\frac12}
    \\
    &\leq \seminorm{\bvec{u},\upw,h}{\rho_h}\left(
    \sum_{F \in \Fh}\norm{\bvec{L}^2(F)}{\bvec{u}_F} \norm{L^4(F)}{\jump{\pi_h\varphi}}^2
    \right)^{\frac12}
    \\
    &\leq \seminorm{\bvec{u},\upw,h}{\rho_h}\left(
    \sum_{F \in \Fh} h_F \norm{\bvec{L}^2(F)}{\bvec{u}_F}^2
    \right)^{\frac14} \left(
    \sum_{F \in \Fh} h_F^{-1} \norm{L^4(F)}{\jump{\pi_h\varphi - \varphi}}^4
    \right)^{\frac14}
    \\
    &\lesssim h^{\frac12} \seminorm{\bvec{u},\upw,h}{\rho_h}
    \norm{a,h}{\uvec{u}_h}^{\frac12}
    \seminorm{W^{1,4}(\Omega)}{\varphi},
  \end{aligned}
  \]
  where we used the Cauchy--Schwarz inequality in the first step, %
  the definition~\eqref{eq:dh:partial coercivity} of the upwind seminorm and H\"older inequalities along with $\norm{\bvec{L}^\infty(F)}{\normal_F} \le 1$ in the second step, %
  the Cauchy--Schwarz inequality for sums in the third step,
  while, in the fourth step, we have used \eqref{eq:discrete.sobolev} followed by \eqref{eq:norm.a.h} for the second factor and the approximation properties of $\pi_h$ (consequence of~\eqref{eq:approximation.mesh.element}) for the third one.

  Rearranging the sums, we write
  \[
  \begin{aligned}
    \res_2
    &= \sum_{F \in \Fh} \int_F (\bvec{u}_F \cdot \normal_F) \jump{\rho_h} (\varphi - \average{\pi_h\varphi})
    \\
    &\leq \sum_{F \in \Fh} \norm{L^2(F)}{|\bvec{u}_F \cdot \normal_F|^{\frac12} \jump{\rho_h}} \norm{L^4(F)}{|\bvec{u}_F|^{\frac12}} \norm{L^4(F)}{\varphi - \average{\pi_h\varphi}}
    \\
    &\lesssim \seminorm{\bvec{u}, \upw, h}{\rho_h} \left(
    \sum_{F \in \Fh} h_F \norm{\bvec{L}^2(F)}{\bvec{u}_F}^2
    \right)^{\frac14} \left(
    \sum_{F \in \Fh} h_F^{-1} \norm{L^4(F)}{\varphi - \average{\pi_h\varphi}}^4
    \right)^{\frac14}
    \\
    \overset{\eqref{eq:discrete.sobolev},\eqref{eq:norm.a.h}}&\lesssim h^{\frac12} \seminorm{\bvec{u}, \upw, h}{\rho_h} \norm{a,h}{\uvec{u}_h}^{\frac12} \seminorm{W^{1,4}(\Omega)}{\varphi},
  \end{aligned}
  \]
  where we have used the fact that $\norm{L^4(F)}{\varphi - \average{\pi_h \varphi}}^4 \leq \frac12 \sum_{T \in \TF} \norm{L^4(F)}{\varphi - \pi_T \varphi}$ and concluded with the approximation property~\eqref{eq:approximation.mesh.element}.

  For $\res_3$, we use the divergence-free property of the discrete velocity $\uvec{u}_h\in\Zh$ and the fact that $\sum_{F \in \FT}\int_F \normal_{TF} = 0$ for all $T\in\Th$ to get
  \[
  \begin{aligned}
    \res_3 &=
    \sum_{T \in \Th}\sum_{F \in \FT}\int_F \rho_T (\bvec{u}_T - \bvec{u}_F)\cdot\normal_{TF}(\varphi - \pi_T\varphi)
    \\
    & \leq \upr{\rho} \left( \sum_{T \in \Th} h_T^{-1}\sum_{F \in \FT}\norm{\bvec{L}^2(F)}{\bvec{u}_T - \bvec{u}_F}^2 \right)^{\frac12} \left( \sum_{T \in \Th} h_T\sum_{F \in \FT}\norm{L^2(F)}{\varphi - \pi_T\varphi}^2 \right)^{\frac12}\
    \\
    \overset{\eqref{eq:norm.a.h},\eqref{eq:approximation.mesh.element}}&\lesssim
    h \upr{\rho} \norm{a,h}{\uvec{u}_h} \seminorm{H^1(\Omega)}{\varphi}.
  \end{aligned}
  \]
  Collecting the above estimates yields the desired result with $\res = \res_1 + \res_2 + \res_3$.
\end{proof}

In view of the reformulation of the momentum balance equation discussed in Section~\ref{sec:discrete:reformulation} below (cf., in particular,~\eqref{eq:discrete:momentum.bis}), we need the following specific boundedness result for $d_h$.

\begin{lemma}[Boundedness of $d_h$]\label{lem:dh:boundedness}
  For any $\eta_h \in \Phi_h$ such that $\eta_h \le \upr{\eta}$ almost everywhere in $\Omega$ and all
  $(\uvec{v}_h, \uvec{w}_h, \uvec{z}_h) \in \Zh \times \UhZ^2$, it holds
  \begin{equation}\label{eq:dh:boundedness}
    d_h(\uvec{v}_h, \eta_h, \bvec{w}_h \cdot \bvec{z}_h)
    \lesssim \upr{\eta} \norm{a,h}{\uvec{v}_h} \norm{a,h}{\uvec{w}_h} \norm{a,h}{\uvec{z}_h}
  \end{equation}
\end{lemma}

\begin{proof}
  Expanding $d_h$ according to~\eqref{eq:dh} with $\chi_h = \bvec{w}_h\cdot\bvec{z}_h$ and noticing that
  \[
  \sum_{T \in \Th} \sum_{F \in \FT}  \int_F \eta_F \,\bvec{v}_F \cdot \normal_{TF}\, (\bvec{w}_F \cdot \bvec{z}_F) = 0
  \]
  since $\eta_F \bvec{v}_F (\bvec{w}_F \cdot \bvec{z}_F)$ is single-valued on internal faces and $\bvec{v}_F \equiv \bvec{0}$ on boundary faces,
  we get
  \[
  \begin{aligned}
    &d_h(\uvec{v}_h,\eta_h,\bvec{w}_h\cdot\bvec{z}_h)
    \\
    &\quad
    \begin{aligned}[t]
      &=
      \sum_{T \in \Th} \sum_{F \in \FT}  \int_F \eta_F \,\bvec{v}_F \cdot \normal_{TF} \, \left[
        (\bvec{w}_T \cdot \bvec{z}_T) - (\bvec{w}_F \cdot \bvec{z}_F)
        \right]
      \\
      &= \frac12 \sum_{T \in \Th} \sum_{F \in \FT}  \int_F \eta_F \,\bvec{v}_F \cdot \normal_{TF} \, \left[
        (\bvec{w}_T - \bvec{w}_F) \cdot (\bvec{z}_T + \bvec{z}_F)
        + (\bvec{w}_T + \bvec{w}_F) \cdot (\bvec{z}_T - \bvec{z}_F)
        \right]
      \\
      &\le
      \upr{\eta}
      \left(
      \sum_{T \in \Th} \sum_{F \in \FT} h_T \norm{\bvec{L}^4(F)}{\bvec{v}_F}^4
      \right)^{\frac14}
      \norm{1,h}{\uvec{w}_h}
      \left(
      \sum_{T \in \Th} \sum_{F\in\FT}h_T \norm{\bvec{L}^4(F)}{\bvec{z}_T + \bvec{z}_F}^4
      \right)^{\frac14}
      \\
      &\quad
      + \upr{\eta}
      \left(
      \sum_{T \in \Th} \sum_{F \in \FT} h_T \norm{\bvec{L}^4(F)}{\bvec{v}_F}^4
      \right)^{\frac14}
      \left(
      \sum_{T \in \Th} \sum_{F \in \FT}h_T \norm{\bvec{L}^4(F)}{\bvec{w}_T + \bvec{w}_F}^4
      \right)^{\frac14}
      \norm{1,h}{\uvec{z}_h},
    \end{aligned}
  \end{aligned}
  \]
  where we have used H\"{o}lder inequalities together with $\norm{\bvec{L}^\infty(F)}{\normal_F} \le 1$ and the definition~\eqref{eq:norm.1.h} of the $\norm{1,h}{{\cdot}}$-norm in the third step.
  We next use~\eqref{eq:discrete.sobolev} with $p = 4$ to estimate the factors involving $\bvec{v}_F$
  and triangle inequalities together with the relation $(a + b)^4 \lesssim a^4 + b^4$ followed by~\eqref{eq:discrete.sobolev} with $p = 4$ to estimate the factors involving $\bvec{w}_T + \bvec{w}_F$ and $\bvec{z}_T + \bvec{z}_F$, thus arriving at
  $
  d_h(\uvec{v}_h,\eta_h,\uvec{w}_h\cdot\uvec{z}_h)
  \lesssim \upr{\eta}
  \norm{1,h}{\uvec{v}_h}
  \norm{1,h}{\uvec{w}_h}
  \norm{1,h}{\uvec{z}_h}.
  $
  Recalling the norm equivalence in~\eqref{eq:norm.a.h} yields~\eqref{eq:dh:boundedness}.
\end{proof}

\subsection{Space semi-discrete problem and main results}

The space semi-discrete scheme stated in this section is inspired by~\cite{Guermond.Quartapelle:00}, where stability is achieved through a specific formulation of the unsteady term in the momentum balance equation based on the auxiliary variable $\sigma_h\coloneqq \rho_h^{\frac12}$ (the fact that $\sigma_h$ is real-valued is a consequence of the discrete maximum principle~\eqref{eq:maximum.principle} below).
Specifically, we consider the following problem: Find $(\rho_h,\uvec{u}_h) \in C^1([0,\tF];\Phi_h) \times C^1([0,\tF]; \Zh)$ such that, for all $t\in\lparen 0, \tF \rbrack$ and all $(\chi_h, \uvec{v}_h) \in \Phi_h \times \Zh$,
\begin{subequations}\label{eq:discrete}
  \begin{align}
    \label{eq:discrete:density}
    \int_\Omega\partial_t\rho_h\, \chi_h
    + d_h(\uvec{u}_h, \rho_h, \chi_h) &= 0,
    \\ \label{eq:discrete:momentum}
    \int_{\Omega}\sigma_h\partial_t(\sigma_h\bvec{u}_h)\cdot\bvec{v}_h
    + \lwr{\rho} j_h(\partial_t\uvec{u}_h,\uvec{v}_h)
    + \mu a_h(\uvec{u}_h, \uvec{v}_h)
    + c_h(\underline{(\rho\bvec{u})}_h, \uvec{u}_h, \uvec{v}_h)
    &= \int_\Omega \bvec{f} \cdot \bvec{v}_h.
  \end{align}
  and
  \begin{equation}\label{eq:discrete:ic}
    \text{
      $\uvec{u}_h(0) = \Ih\bvec{u}^0$
      and $\rho_h(0) = \pi_h \rho^0$.
    }
  \end{equation}
\end{subequations}
\begin{remark}[Zero-divergence constraint]
  In practice, the zero-divergence constraint embedded in the velocity space $\Zh$ defined by~\eqref{eq:Zh} can be enforced through a piecewise-constant pressure $p_h$ with zero mean-value over $\Omega$.
  The fact that the resulting velocity-pressure coupling is inf-sup stable is classical (see, e.g.,~\cite{Aghili.Boyaval.ea:15} and also~\cite[Chapter~8]{Di-Pietro.Droniou:20}) and stems from the fact that $\Dh \Ih \bvec{v} = \pi_h (\DIV \bvec{v})$ for all $\bvec{v} \in \bvec{H}^1(\Omega)$ together with the boundedness of the interpolator.
\end{remark}
The main results of our analysis are stated below.
\begin{lemma}[A priori estimates]\label{lem:a.priori}
  Let $(\rho_h, \uvec{u}_h)$ solve \eqref{eq:discrete}. Then, recalling \eqref{eq:cont.maximum.principle.density}, it holds,
  \begin{equation}\label{eq:maximum.principle}
    \lwr{\rho}\leq \rho_h(\bvec{x},t)\leq \upr{\rho}\quad\text{for almost every $(\bvec{x},t)\in\Omega\times(0,\tF)$}
  \end{equation}
  and
  \begin{subequations}\label{eq:a.priori}
    \begin{align}\label{eq:a.priori.density}
      \max_{t\in[0,\tF]}\left(
      \norm{L^2(\Omega)}{\rho_h(t)}^2
      + \int_0^{t} \seminorm{\bvec{u},\upw,h}{\rho_h(s)}^2\, ds
      \right)
      &\le \norm{L^2(\Omega)}{\rho^0}^2,\\
      \label{eq:a.priori.velocity}
      \lwr{\rho}\max_{t\in[0,\tF]}\norm{0,h}{\uvec{u}_h(t)}^2
      + \mu \norm{L^2(0,\tF;\,\UhZ)}{\uvec{u}_h}^2
      &\lesssim e^{\tF}\left(
      \norm{L^2(0,\tF; \bvec{L}^2(\Omega))}{\bvec{f}}^2
      + \upr{\rho} \seminorm{\bvec{H}^1(\Omega)}{\bvec{u}^0}^2
      \right).
    \end{align}
  \end{subequations}
\end{lemma}

\begin{proof}
  See Section~\ref{sec:stability:a.priori}.
\end{proof}
\begin{theorem}[Existence and uniqueness]\label{thm:existence.uniqueness}
  There exists a unique solution $(\rho_h,\uvec{u}_h)\in C^1(\lbrack 0,\tF\rparen;\Phi_h\times\Zh)$ to \eqref{eq:discrete}.
\end{theorem}

\begin{proof}
  See Section~\ref{sec:stability:existence.uniqueness}.
\end{proof}
\begin{theorem}[Convergence to minimal regularity solutions]\label{thm:convergence}
  Let $(\Th)_{h\in\mathcal{H}}$ denote a regular mesh sequence in the sense of~\cite[Definition~1.9]{Di-Pietro.Droniou:20}, and let $((\rho_h,\uvec{u}_h))_{h\in\mathcal{H}}$ be such that, for all $h\in\mathcal{H}$, $(\rho_h,\uvec{u}_h)$ solves \eqref{eq:discrete}. Then, there exists a weak solution $(\rho,\bvec{u})$ of \eqref{eq:continuous.problem} such that, up to a subsequence as $h\to 0$,
  \begin{enumerate}
  \item $\bvec{u}_h \rightarrow \bvec{u}$ strongly in $L^p(0,\tF;\bvec{L}^2(\Omega))$ for all $p\in[1,\infty)$ and,
    recalling~\eqref{eq:q*}, $\bvec{u}_h \rightharpoonup \bvec{u}$ weakly in $L^2(0,\tF;\bvec{L}^q(\Omega))$ for all $q\in[1,q^*)$;

    \item $\Gh\uvec{u}_h \rightharpoonup \GRAD\bvec{u}$ weakly in $L^2(0,\tF;\bbvec{L}^2(\Omega))$;

    \item $\rho_h \overset{\ast}{\rightharpoonup} \rho$ weakly-$\ast$ in $L^{\infty}(0,\tF;L^{\infty}(\Omega))$ and
      $\rho_h \rightarrow \rho$ strongly in $L^p(0,\tF;L^p(\Omega))$ for all $p\in[1,\infty)$.
  \end{enumerate}

\end{theorem}

\begin{proof}
  See Section~\ref{sec:convergence:proof}.
\end{proof}

\subsection{Numerical illustration}

% \begin{table}[b!]
% 	\centering
% 	\begin{tabular}{c|c|c}
% 		Triangular mesh & Cartesian mesh & Hexagonal mesh \\
% 		$h_0 = 0.25$ & $h_0 = 0.2$ & $h_0 \approx 0.12$ \\
% 		$h_1 = 0.125$ & $h_1 = 0.1$ & $h_1 \approx 0.06$ \\
% 		$h_2 = 0.0625$ & $h_2 = 0.05$ & $h_2 \approx 0.03$ \\
% 		$h_3 = 0.03125$ & $h_3 = 0.025$ & $h_3 \approx 0.01$ 
% 	\end{tabular}
% 	\caption{Mesh sizes at each refinement step of the convergence analysis for each type of mesh}
% 	\label{tab:mesh.sizes}
% \end{table}

To validate the theoretical results and experimentally investigate the orders of convergence, we consider the following two-dimensional manufactured solution on the unit square $\Omega = (0,1)^2$ and for a final time $\tF = 1$ (see~\cite{Guermond.Quartapelle:00}):
\begin{equation}\label{eq:guermond.solution}
  \rho(x,y,t) = 2 + x \cos(\sin t) + y \sin(\sin t),
  \qquad
  \bvec{u}(x,y,t) = \begin{pmatrix}
    -y \\ x
  \end{pmatrix} \cos t.
\end{equation}
The fully-discrete problem is obtained by using a staggered scheme.
The space domain is discretised by means of three mesh families, triangular, Cartesian, and (mostly) hexagonal, each with four refinement levels.
Each mesh family starts with an initial mesh size $h_0$ (see Figure~\ref{fig:meshes}), which is exactly halved at each refinement for the triangular and Cartesian meshes, and approximately halved for the hexagonal meshes.
We subdivide the time interval by considering an initial time step $\delta t_0 = 10^{-3}$, and then halve it as well at each refinement level.

Figure~\ref{fig:convergence.analysis} display the behaviour of the errors for the density and the velocity in the norms defined by the left-hand sides of \eqref{eq:a.priori.density} and \eqref{eq:a.priori.velocity}, respectively.
According to known a priori estimates for HHO discretisations of the incompressible Navier--Stokes equations in the constant density case (cf., e.g.,~\cite{Di-Pietro.Krell:17,Botti.Di-Pietro.ea:19*1} and also~\cite[Chapter~9]{Di-Pietro.Droniou:20}), we can expect the velocity to converge in $h$ (for smooth solutions and aside from possible superconvergence phenomena).
In the present setting, however, the convergence could be actually less, due to the possibly weaker approximation of the density.
As a matter of fact, for the latter, the classical theory of discontinuous Galerkin methods for advection equations would suggest convergence in $h^{\frac12}$ (see, e.g., \cite[Chapters~2 and~3]{Di-Pietro.Ern:12}).
In practice, we indeed observe convergence in $h^{\frac12}$ for the triangular mesh family, while higher convergence rates are observed in the other cases, particularly for the Cartesian mesh, possibly due to a superconvergence phenomenon on structured grids.
\begin{figure}[t!]
	\hfill
	\subfigure[Triangular mesh with $h_0 = 0.25 $]{\includegraphics[scale=0.35]{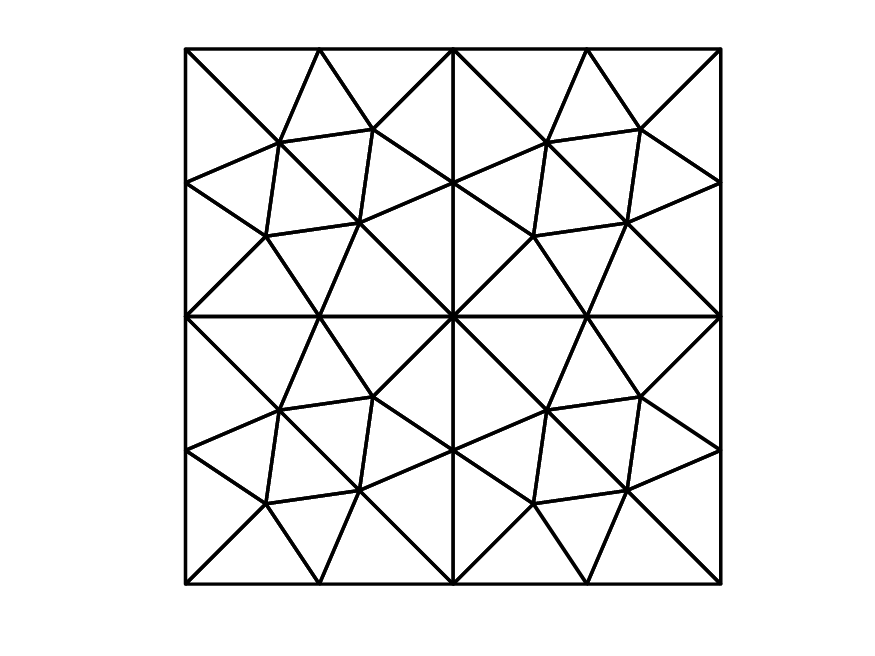}\label{fig:mesh.tri}}
	\hfill
	\subfigure[Cartesian mesh with $h_0 = 0.2$]{\includegraphics[scale=0.35]{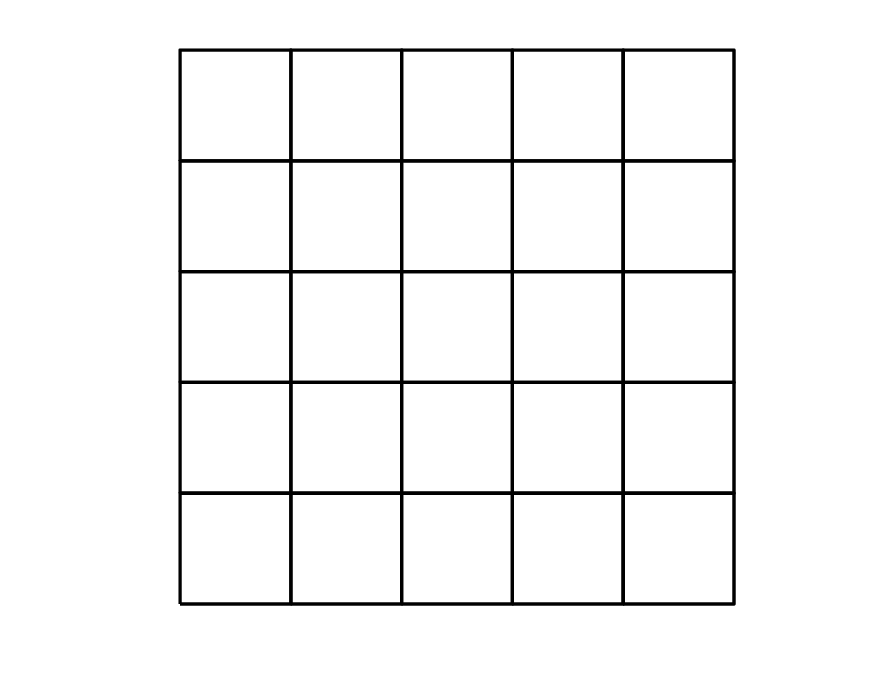}\label{fig:mesh.cart}}
	\hfill
	\subfigure[Hexagonal mesh with $h_0 \approx 0.125$]{\includegraphics[scale=0.35]{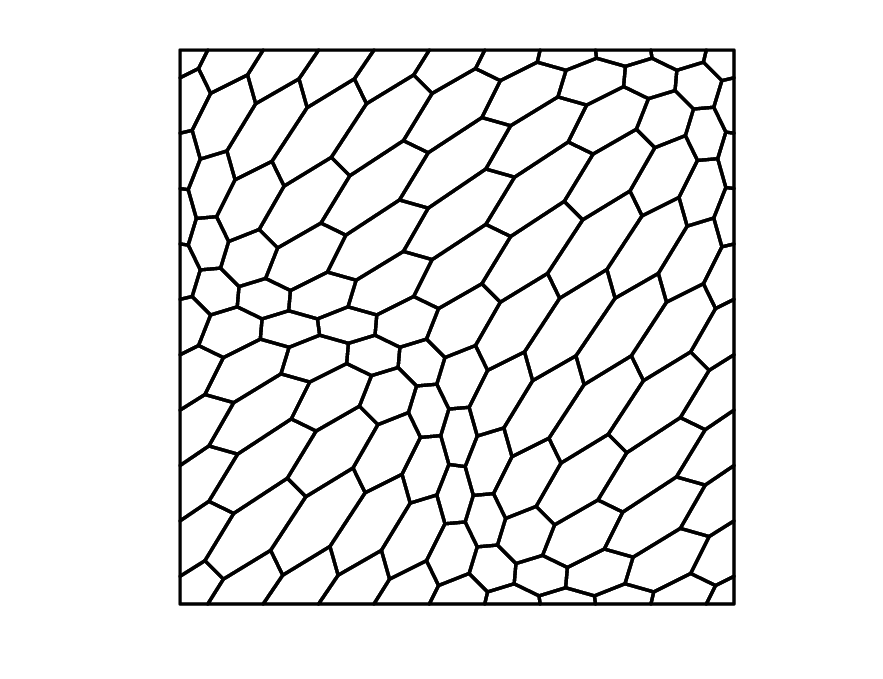}\label{fig:mesh.hexa}}
	\hfill
	\caption{Illustration of the meshes used at the initial refinement step of the convergence analysis}
	\label{fig:meshes}
\end{figure}
\begin{figure}[t!]
	\hfill
	\subfigure[Density error in energy norm]{\includegraphics[scale=0.18]{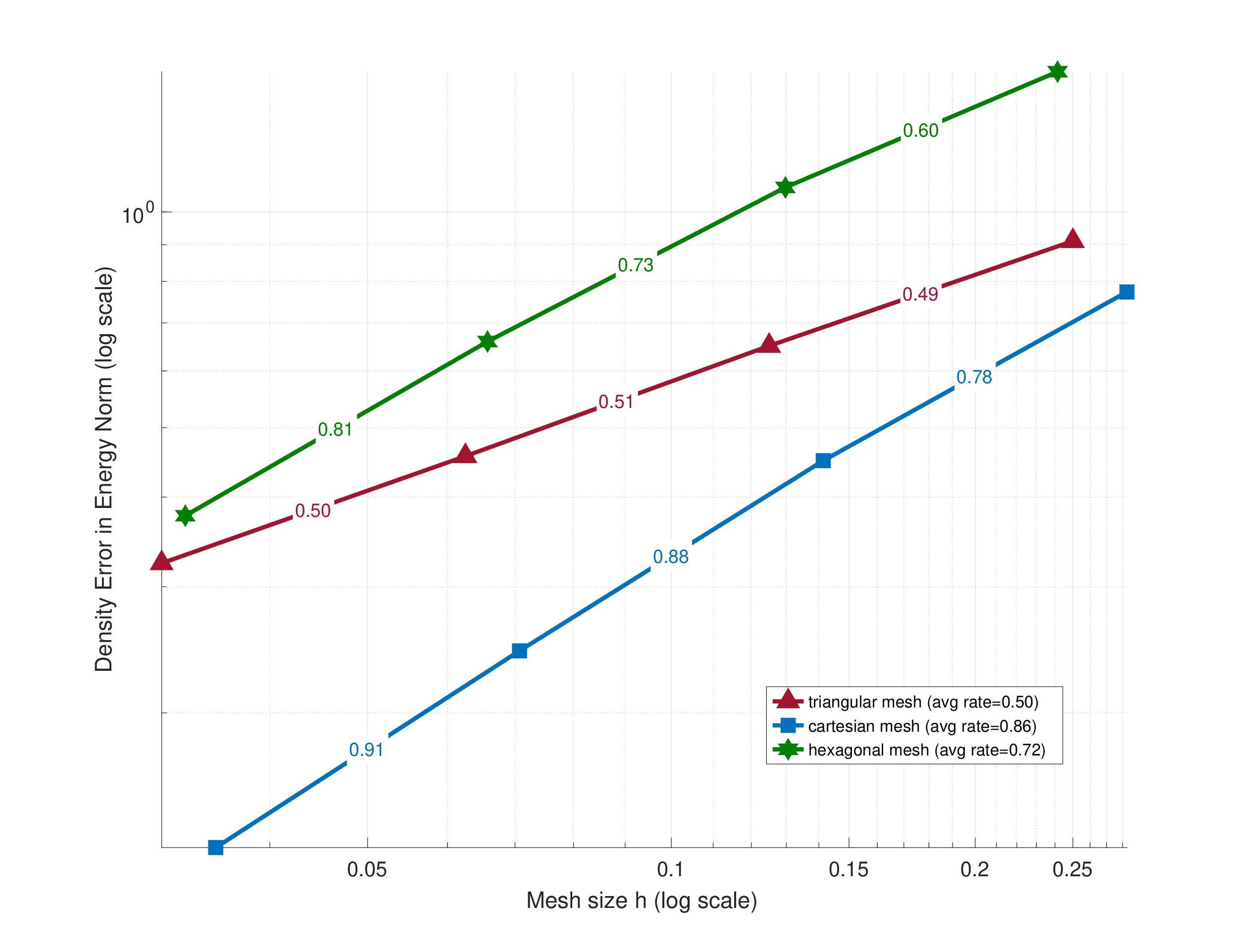}\label{fig:convergence.analysis.density}}
	\hfill
	\subfigure[Velocity error in energy norm]{\includegraphics[scale=0.18]{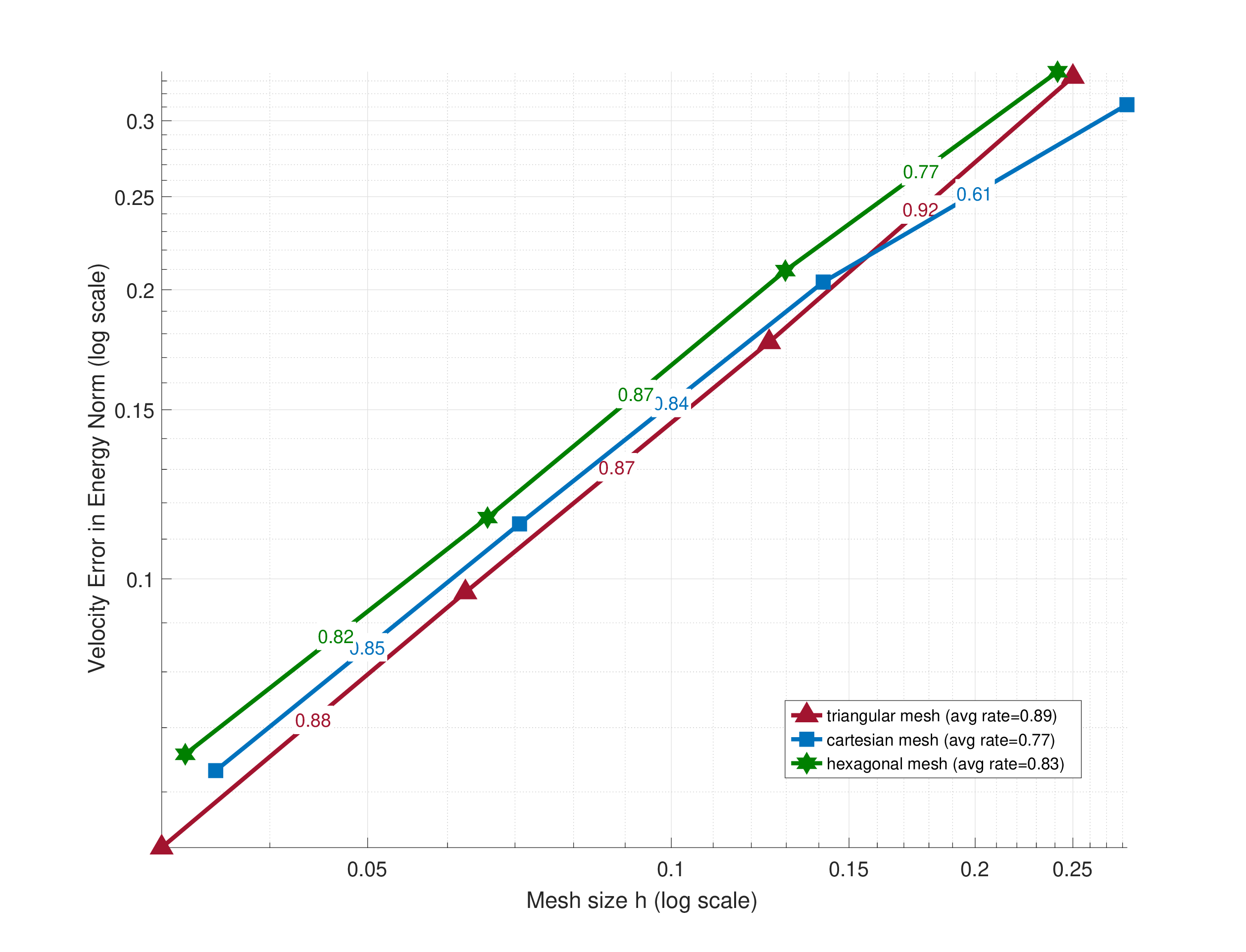}\label{fig:convergence.analysis.velocity}}
	\hfill
	\caption{Convergence analysis of density and velocity to solution \eqref{eq:guermond.solution} for three mesh types (triangular, Cartesian and hexagonal depicted in Figure~\ref{fig:meshes})}
	\label{fig:convergence.analysis}
\end{figure}

\subsection{A reformulation of discrete momentum balance with a view toward consistency}\label{sec:discrete:reformulation}

The formulation~\eqref{eq:discrete:momentum} of the discrete momentum balance equation is the natural one for stability analysis, and will be used in Section~\ref{sec:stability} to derive a priori estimates on the solution.
In order to bridge it with its counterpart~\eqref{eq:weak:momentum} in the weak formulation, we can notice that
\begin{equation}\label{eq:discrete:momentum.bis:intermediate}
  \int_\Omega \sigma_h\partial_t (\sigma_h\bvec{u}_h)\cdot\bvec{v}_h
  =
  \int_\Omega \partial_t (\rho_h \bvec{u}_h) \cdot \bvec{v}_h
  - \frac12 \int_\Omega \partial_t\rho_h \, (\bvec{u}_h \cdot \bvec{v}_h)
  \overset{\eqref{eq:discrete:density}}=
  \int_\Omega \partial_t (\rho_h \bvec{u}_h) \cdot \bvec{v}_h
  + \frac12 d_h(\uvec{u}_h,\rho_h,\bvec{u}_h\cdot\bvec{v}_h),
\end{equation}
where the first equality follows from the Leibniz product rule.
We deduce the following equivalent formulation of \eqref{eq:discrete:momentum}:
\begin{multline}\label{eq:discrete:momentum.bis}
  \int_\Omega \partial_t (\rho_h \bvec{u}_h) \cdot \bvec{v}_h
  + \lwr{\rho} j_h(\partial_t\uvec{u}_h,\uvec{v}_h)
  + \mu a_h(\uvec{u}_h, \uvec{v}_h)
  \\
  + \left[
    c_h(\underline{(\rho\bvec{u})}_h, \uvec{u}_h, \uvec{v}_h)
    + \frac12 d_h(\uvec{u}_h, \rho_h, \bvec{u}_h\cdot\bvec{v}_h)
    \right]
  = \int_\Omega \bvec{f} \cdot \bvec{v}_h
  \qquad \forall \uvec{v}_h \in \Zh.
\end{multline}
\begin{remark}[Higher-order velocity approximation]
  The use of~\eqref{eq:discrete:density} in the second step of~\eqref{eq:discrete:momentum.bis:intermediate} is possible since $\bvec{u}_h \cdot \bvec{v}_h \in \Phi_h$.
  While this fact that does not generalise in a straightforward manner to higher orders, it would still be possible to carry out this step for a HHO space one order higher with face unknowns in $\Poly{1}(F)^d$ and (depleted) element unknowns in $\Poly{0}(T)^d$; see \cite[Section~5.1]{Di-Pietro.Droniou:20}.
\end{remark}

The fact that term in brackets plays the role of $-\int_\Omega \rho \bvec{u} \otimes \bvec{u} \colon \GRAD \bvec{v}$ in~\eqref{eq:weak:momentum} is made clear by the following result.
\begin{lemma}[Consistency of $c_h$]\label{lem:ch:consistency}
  For all $\rho_h \in \Phi_h$ such that $\rho_h \le \upr{\rho}$ almost everywhere in $\Omega$ and all $(\uvec{u}_h,\bvec{v}) \in \Zh \times \Cc(\Omega)$, letting $\huvec{v}_h \coloneqq \Ih \bvec{v}$, it holds
  \begin{equation}\label{eq:ch:consistency}
    c_h(\underline{(\rho\bvec{u})}_h, \uvec{u}_h, \huvec{v}_h)
    + \frac12 d_h(\uvec{u}_h, \rho_h(t), \bvec{u}_h \cdot \hvec{v}_h)
    = - \int_{\Omega} (\rho_h \bvec{u}_h \otimes \bvec{u}_h) : \Gh \huvec{v}_h
    + \res,
  \end{equation}
  where
  \begin{equation}\label{eq:ch:consistency:residual}
    |\res| \lesssim 
    h \upr{\rho}\norm{a,h}{\uvec{u}_h}^2 \seminorm{\bvec{W}^{1,4}(\Omega)}{\bvec{v}}
    + h^{\frac12}\seminorm{\bvec{u},\upw,h}{\rho_h} \norm{0,h}{\uvec{u}_h}^{\frac12} \norm{a,h}{\uvec{u}_h} \seminorm{\bvec{W}^{1,\infty}(\Omega)}{\bvec{v}}.
  \end{equation}
\end{lemma}

The proof of Lemma~\ref{lem:ch:consistency} hinges on the following discrete integration by parts.

\begin{proposition}[Discrete integration by parts]\label{prop:discrete.ibp}
  For any $(\rho_h, \uvec{u}_h, \uvec{v}_h) \in \Phi_h \times \UhZ^2$, it holds, with $\rho_F = \rho_F(\uvec{u}_h)$ obtained according to~\eqref{eq:eta.F},
  \begin{equation}\label{eq:discrete.ibp}
    \begin{aligned}
      c_h(\underline{(\rho\bvec{u})}_h,\uvec{u}_h, \uvec{v}_h)
      + \frac12 d_h({}&\uvec{u}_h, \rho_h, \bvec{u}_h \cdot \bvec{v}_h)
      = -\int_{\Omega} (\rho_h \bvec{u}_h \otimes \bvec{u}_h) : \Gh \uvec{v}_h
      \\
      &
      + \frac12 \sum_{T \in \Th} \sum_{F \in \FT} \int_F \rho_F (\bvec{u}_F \cdot \normal_{TF}) \, (\bvec{u}_F - \bvec{u}_T) \cdot (\bvec{v}_T - \bvec{v}_F)
      \\
      &
      + \sum_{T \in \Th} \sum_{F \in \FT} \int_F \rho_T (\bvec{u}_F - \bvec{u}_T) \cdot \normal_{TF} \, \bvec{u}_T \cdot (\bvec{v}_T - \bvec{v}_F)
      \\
      &
      + \sum_{T \in \Th} \sum_{F \in \FT} \int_F (\rho_F - \rho_T) (\bvec{u}_F \cdot \normal_{TF}) \, \bvec{u}_T \cdot (\bvec{v}_T - \bvec{v}_F).
    \end{aligned}
  \end{equation}
\end{proposition}

\begin{proof}
  Denote by $\mathfrak{L}$ the left-hand side of~\eqref{eq:discrete.ibp}.
  Using the definitions~\eqref{eq:ch} of $c_h$ and~\eqref{eq:dh} of $d_h$, we have, after rearranging,
  \[
  \mathfrak{L}
  = \frac12 \sum_{T \in \Th} \sum_{F \in \FT} \int_F \rho_F (\bvec{u}_F \cdot \normal_{TF}) (\bvec{u}_F \cdot \bvec{v}_T)
  + \frac12 \sum_{T \in \Th} \sum_{F \in \FT} \int_F \rho_F (\bvec{u}_F \cdot \normal_{TF}) \, \bvec{u}_T \cdot (\bvec{v}_T - \bvec{v}_F).
  \]
  Subtracting
  $
  \frac12 \sum_{T \in \Th} \sum_{F \in \FT} \int_F \rho_F (\bvec{u}_F \cdot \normal_{TF}) (\bvec{u}_F \cdot \bvec{v}_F) = 0
  $
  (since $\rho_F \bvec{u} _F (\bvec{u}_F \cdot \bvec{v}_F)$ is single-valued on interfaces and vanishes on boundary faces), we go on writing
  \begin{equation}\label{eq:discrete.ibp:intermediate:1}
    \begin{aligned}
      \mathfrak{L}
      &= \frac12 \sum_{T \in \Th} \sum_{F \in \FT} \int_F \rho_F (\bvec{u}_F \cdot \normal_{TF}) \, \bvec{u}_F \cdot (\bvec{v}_T - \bvec{v}_F)
      + \frac12 \sum_{T \in \Th} \sum_{F \in \FT} \int_F \rho_F (\bvec{u}_F \cdot \normal_{TF}) \, \bvec{u}_T \cdot (\bvec{v}_T - \bvec{v}_F)
      \\
      &= \frac12 \sum_{T \in \Th} \sum_{F \in \FT} \int_F \rho_F (\bvec{u}_F \cdot \normal_{TF}) \, (\bvec{u}_F - \bvec{u}_T) \cdot (\bvec{v}_T - \bvec{v}_F)
      + \sum_{T \in \Th} \sum_{F \in \FT} \int_F \rho_F (\bvec{u}_F \cdot \normal_{TF}) \, \bvec{u}_T \cdot (\bvec{v}_T - \bvec{v}_F).
    \end{aligned}
  \end{equation}
  We next notice that
  \begin{equation}\label{eq:discrete.ibp:intermediate:2}
    \begin{aligned}
      \int_\Omega (\rho_h \bvec{u}_h \otimes \bvec{u}_h) : \Gh \uvec{v}_h
      \overset{\eqref{eq:Gh.Dh}}&=
      \sum_{T \in \Th} \int_T (\rho_T \bvec{u}_T \otimes \bvec{u}_T) : \GT \uvec{v}_T
      \\
      \overset{\eqref{eq:GT}}&=
      \sum_{T \in \Th} \sum_{F \in \FT} \int_F (\rho_T \bvec{u}_T \cdot \normal_{TF}) \, \bvec{u}_T \cdot (\bvec{v}_F - \bvec{v}_T),
    \end{aligned}
  \end{equation}
  where we have used the fact that $\bvec{\tau} = \rho_T \bvec{u}_T \otimes \bvec{u}_T \in \Poly{0}(T)^{d \times d}$ is a valid test function in~\eqref{eq:GT}.
  To conclude, we add~\eqref{eq:discrete.ibp:intermediate:2} to~\eqref{eq:discrete.ibp:intermediate:1} and notice that $\rho_F \bvec{u}_F - \rho_T \bvec{u}_T = \rho_T (\bvec{u}_F - \bvec{u}_T) + (\rho_F - \rho_T) \bvec{u}_F$.
\end{proof}

\begin{proof}[Proof of Lemma~\ref{lem:ch:consistency}]
  Using the discrete integration by parts formula \eqref{eq:discrete.ibp} with $\uvec{v}_h = \huvec{v}_h$, we have
  \[
  c_h(\underline{(\rho\bvec{u})}_h, \uvec{u}_h, \huvec{v}_h)
  + \frac12 d_h(\uvec{u}_h, \rho_h(t), \bvec{u}_h \cdot \hvec{v}_h)
  = -\int_{0}^{\tF} \int_{\Omega} (\rho_h \bvec{u}_h \otimes \bvec{u}_h) : \Gh \huvec{v}_h + \res_1 + \res_2 + \res_3,
  \]
  where $\res_1,\res_2,\res_3$ denote the residual terms in~\eqref{eq:discrete.ibp} which we need to bound.
  Applying H\"{o}lder inequalities and using the fact that $\rho_F \le \upr{\rho}$ and $\norm{\bvec{L}^\infty(F)}{\normal_F} \le 1$ for all $F \in \Fh$ as well as $\bvec{\pi}_T \bvec{v} - \bvec{\pi}_F \bvec{v} = \bvec{\pi}_F ( \bvec{\pi}_T \bvec{v} -  \bvec{v} )$ (since $\bvec{\pi}_F$ is linear and it preserves constant functions), we get
  \[
  \begin{aligned}
    \res_1
    &\lesssim \upr{\rho} \sum_{T \in \Th} \sum_{F \in \FT} \norm{\bvec{L}^4(F)}{\bvec{u}_F} \norm{\bvec{L}^2(F)}{\bvec{u}_F - \bvec{u}_T} \norm{\bvec{L}^4(F)}{\bvec{\pi}_F (\bvec{\pi}_T \bvec{v} - \bvec{v})}
    \\
    &\lesssim \upr{\rho} \left(
    \sum_{T \in \Th} \sum_{F \in \FT} h_T \norm{\bvec{L}^4(F)}{\bvec{u}_F}^4
    \right)^{\frac14} \norm{1,h}{\uvec{u}_h} \left(
    \sum_{T \in \Th} \sum_{F \in \FT} h_T \norm{\bvec{L}^4(F)}{\bvec{\pi}_T \bvec{v} - \bvec{v}}^4
    \right)^{\frac14}
    \\
    \overset{\eqref{eq:discrete.sobolev},\eqref{eq:norm.1.h},\eqref{eq:norm.a.h}}&\lesssim
    h \upr{\rho} \norm{a,h}{\uvec{u}_h}^2 \seminorm{\bvec{W}^{1,4}(\Omega)}{\bvec{v}},
  \end{aligned}
  \]
  where we have used a H\"older inequality on the sums along with $1 = h_T^{\frac14} \, h_T^{-\frac12} \, h_T^{\frac14}$ for all $T \in \Th$
  and the $\bvec{L}^4(F)$-boundedness of $\bvec{\pi}_F$ in the second step,
  while the conclusion follows from the approximation properties \eqref{eq:approximation.mesh.element} of $\bvec{\pi}_T$ and $h_T \le h$.

  The fact that $\res_2 \lesssim h \norm{a,h}{\uvec{u}_h}^2 \seminorm{\bvec{W}^{1,4}(\Omega)}{\bvec{v}}$ can be proved in the exact same way, the only difference being that $\norm{\bvec{L}^4(F)}{\bvec{u}_F}$ is replaced by $\norm{\bvec{L}^4(F)}{\bvec{u}_T}$ in the first and second inequalities.

  Finally, to treat $\res_3$, we observe that, by definition \eqref{eq:eta.F} of the upwind value $\rho_F$, the difference $\rho_F - \rho_T$ equals either $\pm\jump{\rho_h}$ or 0.
  Then, applying H\"{o}lder inequalities, we get
  \[
  \begin{aligned}
    \res_3
    &\le
    \sum_{T \in \Th} \sum_{F \in \FT} \norm{L^2(F)}{|\bvec{u}_F \cdot \normal_F|^{\frac12} \jump{\rho_h}} \norm{L^4(F)}{|\bvec{u}_F\cdot\normal_F|^{\frac12}} \norm{\bvec{L}^4(F)}{\bvec{u}_T} \norm{\bvec{L}^\infty(F)}{\bvec{\pi}_F (\bvec{\pi}_T \bvec{v} - \bvec{v})}
    \\
    \overset{\eqref{eq:dh:partial coercivity}}&\lesssim
    \seminorm{\bvec{u},\upw,h}{\rho_h} \left(
    \sum_{T \in \Th} \sum_{F \in \FT} h_T \norm{\bvec{L}^2(F)}{\bvec{u}_F}^2
    \right)^{\frac14} \left( \sum_{T \in \Th} h_T \norm{\bvec{L}^4(\partial T)}{\bvec{u}_T}^4
    \right)^{\frac14}
    \left(
    \max_{T\in\Th} h_T^{-\frac12} \norm{\bvec{L}^\infty(\partial T)}{\bvec{\pi}_T \bvec{v} - \bvec{v}}
    \right)
    \\
    \overset{\eqref{eq:discrete.sobolev}}&\lesssim
    h^{\frac12} \seminorm{\bvec{u},\upw,h}{\rho_h}
    \norm{0,h}{\uvec{u}_h}^{\frac12} \norm{a,h}{\uvec{u}_h}
    \seminorm{\bvec{W}^{1,\infty}(\Omega)}{\bvec{v}},
  \end{aligned}
  \]
  where we have additionally used the $\bvec{L}^\infty(F)$-boudnedndess of $\bvec{\pi}_F$ in the second step and, to conclude,
  the approximation properties of $\bvec{\pi}_T$ together with the bound $h_T \norm{\bvec{L}^2(F)}{\bvec{u}_F}^2\lesssim \norm{0,T}{\uvec{u}_T}$, which can be obtained from the definition \eqref{eq:L2.discrete.product} of the local $L^2$-product using a triangle inequality and a discrete trace inequality (recall also that $\bvec{u}_F=\bvec{0}$ for all $F\in\Fhi$).
\end{proof}

%------------------------------------------------------------------------------%

\section{A priori estimates, existence, and uniqueness}\label{sec:stability}

\subsection{Proof of Lemma~\ref{lem:a.priori}}\label{sec:stability:a.priori}

\begin{proof}[Proof of Lemma~\ref{lem:a.priori}]
  Let $t\in (0,\tF)$ and set $M(t)\coloneqq \max_{T'\in\Th}\rho_{T'}(t)$. Since each $\rho_{T'}$ is Lipschitz in time, as a solution of an ODE with bounded coefficients, the mapping $t\to M(t)$ is differentiable in the sense of subdifferential calculus \cite[Theorem 10.31]{Rockafellar.Wets:98} and $\partial_t M(t)$ is contained in the convex hull of
  \[
  \mathfrak A(t)\coloneq \left\{\partial_t\rho_T(t)\st T\in\Th\text{ is such that }\rho_T(t)=M(t) \right\}.
  \]
  Let $T\in\Th$ be such that $\rho_T(t)=M(t)$. Then, for all $F\in \FT\cap\Fhi$, denoting by $T'$ the cell on the other side of $F$ from $T$,
  \begin{equation}\label{eq:upper.bound.rho.T}
    \rho_{T'}(t)\leq M(t).
  \end{equation}
  Choose now $\chi_h$ in \eqref{eq:discrete:density} such that $\chi_T \equiv 1$ and $\chi_S \equiv 0$ for all $S \in \Th \setminus \{T\}$, and recall \eqref{eq:dh} and \eqref{eq:fv.upwind.flux} to infer
  \[
  \seminorm{}{T}\partial_t\rho_T = - \sum_{F \in \FT\cap\Fhi}\int_F \left((\bvec{u}_F\cdot\normal_{TF})^+\rho_T - (\bvec{u}_F\cdot\normal_{TF})^-\rho_{T'} \right),
  \]
  where, for the sake of conciseness, we have understood that all functions of time are evaluated at instant $t$.
  Substituting $\rho_T$ with $M(t)$ and using \eqref{eq:upper.bound.rho.T} on the right-hand side, we get
  \begin{align*}
    \seminorm{}{T}\partial_t \rho_T(t) &\leq -\sum_{F \in \FT\cap\Fhi}\int_F \left((\bvec{u}_F\cdot\normal_{TF})^+ M(t) - (\bvec{u}_F\cdot\normal_{TF})^- M(t) \right)
    \\
    & = -M(t)\sum_{F \in \FT\cap\Fhi}\int_F \left( \bvec{u}_F\cdot\normal_{TF}\right)
    \overset{\eqref{eq:DT}}= -M(t)\int_T \DT \uvec{u}_T \overset{\uvec{u}_h\in \Zh}= 0.
  \end{align*}
  where, in the second last equality, we have additionally used the fact that $\bvec{u}_F=\bvec{0}$ whenever $F\in\FT\cap\Fhb$.
  Hence, all numbers in $\mathfrak A(t)$ are nonpositive, and so are all numbers in its convex hull. In particular, this gives $\partial_t M(t)\le 0$,
  and $M$ is therefore non-increasing in time. We deduce that $M(t)\leq M(0) \leq \upr{\rho}$, which establishes the upper bound in \eqref{eq:maximum.principle}. The proof of the lower bound follows similar arguments.
  \smallskip

  Let us prove \eqref{eq:a.priori}. Take $\chi_h = \rho_h$ in \eqref{eq:discrete:density} and use the coercivity \eqref{eq:dh:partial coercivity} of $d_h$ to get
  \begin{equation*} %\label{eq:a.priori:1}
    \partial_t \norm{L^2(\Omega)}{\rho_h}^2 + \seminorm{\bvec{u},\upw,h}{\rho_h}^2 = 0.
  \end{equation*}
  Integrating this relation over $(0,t)$ for an arbitrary $t$, and then passing to the maximum over $t\in[0,\tF]$ proves \eqref{eq:a.priori.density}.

  On the other hand, taking $\uvec{v}_h = \uvec{u}_h$ in \eqref{eq:discrete:momentum}, and recalling the non-dissipativity \eqref{eq:ch:non.dissipativity} of $c_h$, we get
  \[
  \frac12\partial_t\norm{\bvec{L}^2(\Omega)}{\sigma_h\bvec{u}_h}^2
  + \frac12 \lwr{\rho}\partial_t j_h(\uvec{u}_h,\uvec{u}_h)
  + \mu \norm{a,h}{\uvec{u}_h}^2
  = \int_\Omega \bvec{f} \cdot \bvec{u}_h
  \le \frac{1}{2} \norm{\bvec{L}^2(\Omega)}{\bvec{f}}^2
  + \frac12 \norm{\bvec{L}^2(\Omega)}{\bvec{u}_h}^2.
  \]
  Considering that $\norm{\bvec{L}^2(\Omega)}{\bvec{u}_h} \leq \norm{0,h}{\uvec{u}_h}$, integrating over $[0,t]$ for a generic $t\in(0,\tF)$, and recalling the definition~\eqref{eq:space-time.norm} of the time-space norm, we get
  \begin{multline*}
    \norm{\bvec{L}^2(\Omega)}{(\sigma_h\bvec{u}_h)(t)}^2
    + \lwr{\rho}j_h(\uvec{u}_h(t),\uvec{u}_h(t))
    + 2 \mu\norm{L^2(0,t;\,\UhZ)}{\uvec{u}_h}^2\\
    \le
    \norm{\bvec{L}^2(\Omega)}{(\sigma_h\bvec{u}_h)(0)}^2
    + \lwr{\rho}j_h(\uvec{u}_h(0),\uvec{u}_h(0))
    + \norm{L^2(0,t;\bvec{L}^2(\Omega))}{\bvec{f}}^2
    + \int_{0}^t \norm{0,h}{\uvec{u}_h(s)}^2\,ds.
  \end{multline*}
  Noticing that $\rho_h(t) = \sigma_h(t)^2$, and that $\lwr{\rho} \leq \rho_h(t) \leq \upr{\rho}$ by \eqref{eq:maximum.principle}, the inequality above becomes
  \[
  \lwr{\rho}\norm{0,h}{\uvec{u}_h(t)}^2
  + \mu\norm{L^2(0,t;\,\UhZ)}{\uvec{u}_h}^2\lesssim
  \upr{\rho}\norm{0,h}{\uvec{u}_h(0)}^2
  + \norm{L^2(0,t;\bvec{L}^2(\Omega))}{\bvec{f}}^2
  + \int_{0}^t \norm{0,h}{\uvec{u}_h(s)}^2\,ds.
  \]
  Invoking Gr\"{o}nwall's inequality \cite[Proposition 2.1]{Emmrich:99} as well as \eqref{eq:Ih:continuity.0.h}, the estimate \eqref{eq:a.priori.velocity} follows.
\end{proof}

\subsection{Proof of Theorem~\ref{thm:existence.uniqueness}}\label{sec:stability:existence.uniqueness}

We first introduce a definition that will be used in the proof of Theorem \ref{thm:existence.uniqueness} as well in the proof of Lemma \ref{lem:time.translations} below.
For a given $\uvec{w}_h \in \Zh$ and a weight $\omega:\Omega\to (0,\infty)$ such that $\omega\ge \lwr{\omega}$ for some constant $\lwr{\omega}>0$, we define the linear form $\mathcal{L}_\omega[\uvec{w}_h]:\Zh\to\Real$ by: for all $\uvec{v}_h\in\Zh$,
\begin{equation}\label{eq:def.L}
  \mathcal{L}_\omega[\uvec{w}_h](\uvec{v}_h)= \int_{\Omega}\omega\bvec{w}_h\cdot\bvec{v}_h + \lwr{\omega}j_h(\uvec{w}_h,\uvec{v}_h).
\end{equation}

\begin{proof}[Proof of Theorem~\ref{thm:existence.uniqueness}]
  Let us explicitly recast \eqref{eq:discrete} into a finite-dimensional system of ODEs. Define $F_1:\Phi_h\times\Zh\to \Phi_h$ and $F_2:\Phi_h\times\Zh\to\Zh^*$ such that, for all $(\rho_h,\uvec{u}_h)\in \Phi_h\times\Zh$ and all $(\chi_h,\uvec{v}_h)\in\Phi_h\times \Zh$,
  \begin{align*}
    \int_\Omega F_1[\rho_h,\uvec{u}_h]\chi_h&= -d_h(\uvec{u}_h,\rho_h,\chi_h),\\
    F_2[\rho_h,\uvec{u}_h](\uvec{v}_h)&= \int_\Omega \bvec{f} \cdot \bvec{v}_h
    - \mu a_h(\uvec{u}_h, \uvec{v}_h)
    - c_h(\underline{(\rho\bvec{u})}_h, \uvec{u}_h, \uvec{v}_h)
    + \frac12 d_h(\uvec{u}_h, \rho_h, \bvec{u}_h\cdot\bvec{v}_h).
  \end{align*}
  Writing $\sigma_h\partial_t (\sigma_h\bvec{u}_h)=\rho_h\partial_t\bvec{u}_h+\frac12 (\partial_t \rho_h)\bvec{u}_h$ then using, in a similar way as to obtain \eqref{eq:discrete:momentum.bis}, the density equation \eqref{eq:discrete:density} with $\chi_h=\bvec{u}_h\cdot\bvec{v}_h$, one can easily check that \eqref{eq:discrete} is equivalent to
  \begin{align*}
    \partial_t\rho_h(t) &= F_1[\rho_h(t),\uvec{u}_h(t)],\\
    \mathcal{L}_{\rho_h(t)}[\partial_t\uvec{u}_h(t)] &= F_2[\rho_h(t),\uvec{u}_h(t)].
  \end{align*}
  Since $\rho_h(t)\geq \lwr{\rho}>0$, one can check that $\mathcal{L}_{\rho_h(t)}:\Zh\to\Zh^*$ is invertible, and we can therefore recast this system as
  \[
  \partial_t(\rho_h(t),\uvec{u}_h(t)) = \left(F_1[\rho_h(t),\uvec{u}_h(t)],(\mathcal{L}_{\rho_h(t)})^{-1}\circ F_2[\rho_h(t),\uvec{u}_h(t)]\right)
  \eqcolon
  \tilde{F}[\rho_h(t),\uvec{u}_h(t)].
  \]
  This is a first order system of ODEs in explicit form, with a field $\tilde{F}$ that is locally Lipschitz on the open set $\mathcal D_h \coloneqq \left\{\varrho_h\in\Phi_h\,:\,\frac{\lwr{\rho}}{2} < \varrho_h < 2 \upr{\rho} \right\}\times \Zh$. This system therefore has a unique $C^1$ local-in-time solution by the Cauchy--Lipschitz theorem. Furthermore, the a priori estimates \eqref{eq:maximum.principle} and \eqref{eq:a.priori} ensure that the local solution does not reach the boundary of $\mathcal D_h$, and so that it exists up to $\tF$.
\end{proof}

%------------------------------------------------------------------------------%

\section{Convergence}\label{sec:convergence}

\subsection{Estimates on time translations}

We prove in this section an estimate of the variation in time of the approximate velocities, which will be instrumental to proving compactness; see Step~1 in the proof of Theorem~\ref{thm:convergence}.

\begin{lemma}[Estimates on time translations]\label{lem:time.translations}
  If $\uvec{u}_h \in L^2(0,\tF;\Zh)$ is the velocity field solution of \eqref{eq:discrete}, then, for any real number $0 < \delta < \tF$, it holds
  \[
  \int_{ \delta}^{\tF} \norm{\bvec{L}^2(\Omega)}{\bvec{u}_h(t) - \bvec{u}_h(t - \delta)}\,dt
  \lesssim \delta^{\frac12},
  \]
  where the hidden constant is independent of the meshsize and of $\delta$, but possibly depends on the upper bounds in \eqref{eq:a.priori}, on $\upr{\rho}$, on the final time $\tF$, and on the mesh regularity parameter.
\end{lemma}

\begin{proof}
  We follow the ideas of \cite[Proposition 5.2]{Droniou.Eymard:09}.

  Recall the definition \eqref{eq:def.L} of $\mathcal L_{\omega}$ and note that
  \begin{equation}\label{eq:L2.norm.dual}
    \lwr{\omega}\norm{\bvec{L}^2(\Omega)}{\bvec{w}_h}^2
    \overset{\eqref{eq:L2.discrete.norm},\eqref{eq:L2.discrete.product}}\le
    \lwr{\omega}\norm{0,h}{\uvec{w}_h}^2
    \le \mathcal{L}_\omega[\uvec{w}_h](\uvec{w}_h)
    \le \norm{a,h}{\uvec{w}_h} \, \norm{a,h,*}{\mathcal{L}_\omega[\uvec{w}_h]},
  \end{equation}
  where $\norm{a,h,*}{{\cdot}}$ denotes the norm dual to $\norm{a,h}{{\cdot}}$.
  Take $t\in (\delta,\tF)$ and notice that, since $\rho_h(t-\delta)\ge\lwr{\rho}$ by \eqref{eq:maximum.principle} and $\uvec{u}_h(t)-\uvec{u}_h(t-\delta)\in\Zh$,
  \begin{align}
    \lwr{\rho}^{\frac12}&\int_\delta^{\tF} \norm{\bvec{L}^2(\Omega)}{\bvec{u}_h(t) - \bvec{u}_h(t-\delta)}\,dt
    \nonumber\\
    \overset{\eqref{eq:L2.norm.dual}}&\le \int_\delta^{\tF}\norm{a,h}{\bvec{u}_h(t) - \bvec{u}_h(t-\delta)}^{\frac12}
    \norm{a,h,*}{\mathcal{L}_{\rho_h(t-\delta)}[\uvec{u}_h(t) - \uvec{u}_h(t-\delta)]}^{\frac12}
    \,dt
    \nonumber\\
    &\le \left(
    \int_\delta^{\tF} \norm{a,h}{\bvec{u}_h(t) - \bvec{u}_h(t-\delta)}\,dt
    \right)^{\frac12}\left(
    \int_\delta^{\tF} \norm{a,h,*}{\mathcal{L}_{\rho_h(t-\delta)}[\uvec{u}_h(t) - \uvec{u}_h(t-\delta)]}\,dt
    \right)^{\frac12}
    \nonumber\\
    &\le
    \frac{\delta^{\frac12}}{2}\mu^{\frac12}\int_\delta^{\tF} \norm{a,h}{\bvec{u}_h(t) - \bvec{u}_h(t-\delta)}\,dt
    + \frac{1}{2\delta^{\frac12}} \mu^{-\frac12} \int_\delta^{\tF} \norm{a,h,*}{\mathcal{L}_{\rho_h(t-\delta)}[\uvec{u}_h(t) - \uvec{u}_h(t-\delta)]}\,dt
    \nonumber\\
    \overset{\eqref{eq:a.priori.velocity}}&\lesssim
    \delta^{\frac12}
    + \frac{1}{\delta^{\frac12}} \int_\delta^{\tF} \norm{a,h,*}{\mathcal{L}_{\rho_h(t-\delta)}[\uvec{u}_h(t) - \uvec{u}_h(t-\delta)]}\,dt,
    \label{eq:estimate.time.translation:basic}
  \end{align}
  where we have used a Cauchy--Schwarz inequality in the second step and a generalised Young's inequality in the third.
  It only remains to show that the second addend in the right-hand side is $\lesssim \delta^{\frac12}$.
  To this purpose we notice that, for all $t \in [0,\tF]$, setting $z|^t_s \coloneqq z(t)-z(s)$ and using the definition of dual norm, we have
  \begin{equation*}
    \norm{a,h,*}{\mathcal{L}_{\rho_h(t-\delta)}[\uvec{u}_h|_{t-\delta}^t]}
    \overset{\eqref{eq:def.L}}= \sup_{\uvec{v}_h\in\Zh \setminus \{ \uvec{0} \}}
    \frac{%
      \int_{\Omega}\rho_h(t-\delta)\uvec{u}_h|_{t-\delta}^t\cdot\bvec{v}_h
      + \lwr{\rho}j_h(\uvec{u}_h|^t_{t-\delta},\uvec{v}_h)
    }{\norm{a,h}{\uvec{v}_h}}.
  \end{equation*}
  Take $\uvec{v}_h\in\Zh$, write \eqref{eq:discrete:momentum.bis} at time $s$ and integrate over $s\in(t-\delta,t)$ to get
  \begin{equation}
    \label{eq:translation.estimates:basic:2}
    \begin{aligned}
      \int_\Omega(\rho_h \bvec{u}_h)|_{t-\delta}^t {}&\cdot \bvec{v}_h
      + \lwr{\rho} j_h(\uvec{u}_h|^t_{t-\delta},\uvec{v}_h)
      \\
      ={}& \int_{t-\delta}^t\bigg[
        \int_\Omega \bvec{f}(s) \cdot \bvec{v}_h
        - \mu a_h(\uvec{u}_h(s), \uvec{v}_h) - c_h(\underline{(\rho\bvec{u})}_h(s), \uvec{u}_h(s), \uvec{v}_h)
        \\
        &\qquad\qquad 
        - \frac12 d_h(\uvec{u}_h(s), \rho_h(s), \bvec{u}_h(s) \cdot \bvec{v}_h)
        \bigg]\, ds.
    \end{aligned}
  \end{equation}
  Setting $\chi_h = \bvec{u}_h(t)\cdot\bvec{v}_h$ in \eqref{eq:discrete:density} written at a time $s$ and integrating over $s\in (t-\delta,t)$, we obtain, on the other hand,
  \begin{equation}\label{eq:translation.estimates:basic:3}
    \int_\Omega \rho_h|_{t-\delta}^t \bvec{u}_h(t)\cdot\bvec{v}_h
    = \int_{t-\delta}^t d_h(\uvec{u}_h(s), \rho_h(s), \bvec{u}_h(t)\cdot\bvec{v}_h)\, ds.
  \end{equation}
  We next notice that
  \begin{equation}\label{eq:increment.product}
    (\rho_h \bvec{u}_h)|_{t-\delta}^t - \rho_h|_{t-\delta}^t \bvec{u}_h(t)
    = \rho_h(t - \delta) (\bvec{u}_h(t) - \bvec{u}_h(t-\delta)).
  \end{equation}
  Subtracting \eqref{eq:translation.estimates:basic:3} from \eqref{eq:translation.estimates:basic:2} and using \eqref{eq:increment.product}, we obtain
  \begin{equation}\label{eq:translation.estimates:basic}
    \begin{aligned}
      \int_\Omega\rho_h(t - \delta) &\left(\bvec{u}_h(t) - \bvec{u}_h(t-\delta)\right)\cdot \bvec{v}_h
      + \lwr{\rho}j_h(\uvec{u}_h(t) - \uvec{u}_h(t-\delta),\uvec{v}_h)
      \\
      &\quad=
      \int_{t-\delta}^t\bigg[
	\begin{aligned}[t]
	  &
	  \int_\Omega \bvec{f}(s) \cdot \bvec{v}_h
	  - \mu a_h(\uvec{u}_h(s), \uvec{v}_h)
	  - c_h(\underline{(\rho\bvec{u})}_h(s), \uvec{u}_h(s), \uvec{v}_h)
	  \\
	  &- d_h(\uvec{u}_h(s), \rho_h(s), \bvec{u}_h(t)\cdot\bvec{v}_h)
	  - \frac12 d_h(\uvec{u}_h(s), \rho_h(s), \bvec{u}_h(s) \cdot \bvec{v}_h)
	  \bigg]\, ds.
	\end{aligned}
    \end{aligned}
  \end{equation}
  We next proceed to bound the terms $\term_1,\ldots,\term_5$ in the right-hand side of \eqref{eq:translation.estimates:basic}.
  Cauchy--Schwarz inequalities in space give, for the first and second terms,
  \[
  \begin{gathered}
    \term_1
    \le \left(
    \int_{t-\delta}^t
    \norm{\bvec{L}^2(\Omega)}{\bvec{f}(s)}\,ds
    \right) \norm{\bvec{L}^2(\Omega)}{\bvec{v}_h}
    \overset{\eqref{eq:discrete.sobolev},\eqref{eq:norm.a.h}}{\lesssim} \left(
    \int_{t-\delta}^t
    \norm{\bvec{L}^2(\Omega)}{\bvec{f}(s)}\,ds
    \right) \norm{a,h}{\bvec{v}_h},
    \\
    \term_2
    \le \mu \left(\int_{t-\delta}^t
    \norm{a,h}{\uvec{u}_h(s)}\,ds\right) \norm{a,h}{\uvec{v}_h}.
  \end{gathered}
  \]
  For the third term, use the boundedness \eqref{eq:ch:boundedness} of $c_h$ along with the discrete maximum principle~\eqref{eq:maximum.principle} to see that
  \[
  \term_3\lesssim \upr{\rho}\left(\int_{t-\delta}^t\norm{a,h}{\uvec{u}_h(s)}^2\,ds\right)\norm{a,h}{\uvec{v}_h}.
  \]
  To estimate the fourth and fifth terms, we use the boundedness property \eqref{eq:dh:boundedness} of $d_h$ and the discrete maximum principle~\eqref{eq:maximum.principle} to obtain
  \[
  \begin{aligned}
    \term_4&\lesssim \upr{\rho}  \left(\int_{t-\delta}^t \norm{a,h}{\uvec{u}_h(s)}\,ds\right) \norm{a,h}{\uvec{u}_h(t)} \norm{a,h}{\bvec{v}_h},\\
    \term_5&\lesssim \upr{\rho}  \left(\int_{t-\delta}^t\norm{a,h}{\uvec{u}_h(s)}^2\,ds\right) \norm{a,h}{\bvec{v}_h}.
  \end{aligned}
  \]
  Plugging all these estimates into \eqref{eq:translation.estimates:basic} and taking the supremum over $\uvec{v}_h\in\Zh$ such that $\norm{a,h}{\uvec{v}_h}\le 1$ leads to
  \begin{equation}\label{eq:est.star.norm}
    \begin{aligned}
      \norm{a,h,*}{\mathcal{L}_{\rho_h(t-\delta)}[\uvec{u}_h|_{t-\delta}^t]}
      &\lesssim
      \left(
      \int_{t-\delta}^t
      \norm{\bvec{L}^2(\Omega)}{\bvec{f}(s)}\,ds
      \right)
      + \mu \left(\int_{t-\delta}^t
      \norm{a,h}{\uvec{u}_h(s)}\,ds\right) \\
      &\qquad
      + \left(\int_{t-\delta}^t\norm{a,h}{\uvec{u}_h(s)}^2\right)
      + \left(\int_{t-\delta}^t\norm{a,h}{\uvec{u}_h(s)}\,ds\right) \norm{a,h}{\uvec{u}_h(t)}.
    \end{aligned}
  \end{equation}
  \begin{figure}\centering
    \begin{tikzpicture}[scale=4]
      \def\pdelta{0.10}
      \def\ptF{1}
      \draw[->,thick] (0,0)--(1.2*\ptF,0) node[below right=2pt]{$t$};
      \draw[->,thick] (0,0)--(0,1.2*\ptF) node[above left=2pt]{$s$};
      \draw[pattern=north west lines,thick] (\pdelta,0) -- (\ptF,\ptF-\pdelta) -- (\ptF,\ptF) -- (\pdelta,\pdelta) -- cycle;
      \draw (\pdelta,0) node[below=2pt]{$\delta$};
      \draw[dashed] (\ptF,0) node[below=2pt]{$\tF$} -- (\ptF,\ptF-\pdelta);
      \draw[dashed] (0,\pdelta) node[left=2pt]{$\delta$} -- (\pdelta,\pdelta);
      \draw[dashed] (0,\ptF) node[left=2pt]{$\tF$} -- (\ptF,\ptF);
      \draw (0,0) node[below left=2pt] {$0$};
    \end{tikzpicture}
    \caption{Integration domain in \eqref{eq:est.int.delta}.\label{fig:fubini}}
  \end{figure}
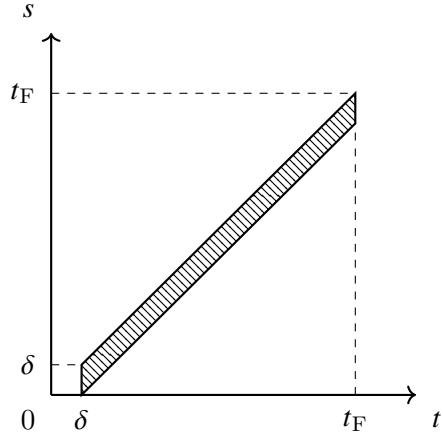
  The Fubini Theorem  shows that, for any function $\varphi \in L^1(0,\tF)$,
  \begin{equation}\label{eq:est.int.delta}
    \int_\delta^{t_F}\left(\int_{t-\delta}^t \varphi(s)\,ds\right)\,dt=\int_0^{t_F}\left(\int_{\max(\delta,s)}^{\min(t_F,s+\delta)}\,dt\right)\varphi(s)\,ds\le \delta\norm{L^1(0,t_F)}{\varphi}
  \end{equation}
  (see Figure~\ref{fig:fubini}).
  Integrating \eqref{eq:est.star.norm} over $t\in (\delta,\tF)$, using \eqref{eq:est.int.delta} with, respectively, $\varphi(s)=\norm{\bvec{L}^2(\Omega)}{\bvec{f}(s)}$, $\varphi(s)=\norm{a,h}{\uvec{u}_h(s)}$, and $\varphi(s)=\norm{a,h}{\uvec{u}_h(s)}^2$, and recalling the bound \eqref{eq:a.priori.velocity} gives
  \begin{equation}\label{eq:est.star.norm.2}
    \int_\delta^{\tF} \norm{a,h,*}{\mathcal{L}_{\rho_h(t-\delta)}[\uvec{u}_h|_{t-\delta}^t]}\,dt
    \lesssim \delta
    + \int_\delta^{\tF}\left(\int_{t-\delta}^t\norm{a,h}{\uvec{u}_h(s)}\,ds\right) \norm{a,h}{\uvec{u}_h(t)}\,dt.
  \end{equation}
  For the last term involving $s$ and $t$, we write, using successive Cauchy--Schwarz inequalities,
  \[
  \begin{aligned}
    &\int_\delta^{t_F} \left(
    \int_{t-\delta}^t \norm{a,h}{\uvec{u}_h(s)}\,ds
    \right)\norm{a,h}{\uvec{u}_h(t)}\,dt
    \\
    &\quad
    \begin{aligned}[t]
      &\le \left[
        \int_\delta^{t_F}\left(\int_{t-\delta}^t \norm{a,h}{\uvec{u}_h(s)}\,ds\right)^2
        \right]^{\frac12}
      \norm{L^2(0,t_F;\,\UhZ)}{\uvec{u}_h}
      \\
      &\le\left[
        \int_\delta^{t_F}\delta\int_{t-\delta}^t \norm{a,h}{\uvec{u}_h(s)}^2\,ds
        \right]^{\frac12}
      \norm{L^2(0,t_F;\,\UhZ)}{\uvec{u}_h}
      \\
      \overset{\eqref{eq:est.int.delta}}&\le
      \left[
        \delta^2\int_0^{t_F} \norm{a,h}{\uvec{u}_h(s)}^2\,ds
        \right]^{\frac12}
      \norm{L^2(0,t_F;\,\UhZ)}{\uvec{u}_h}
      \overset{\eqref{eq:a.priori.velocity}}\lesssim\delta.
    \end{aligned}
  \end{aligned}
  \]
  Plugging this estimate into \eqref{eq:est.star.norm.2} shows that, as required, the last term in \eqref{eq:estimate.time.translation:basic} is $\lesssim\delta^{\frac12}$, thereby concluding the proof.
\end{proof}

\subsection{Proof of Theorem~\ref{thm:convergence}}\label{sec:convergence:proof}

\begin{proof}[Proof of Theorem~\ref{thm:convergence}]
  The proof proceeds in several steps.
  In \textbf{Step 1} we prove the existence of a limit for the sequence of discrete solutions.
  In \textbf{Step 2} we show that this limit solves the density advection equation \eqref{eq:weak:density}.
  In \textbf{Step 3} we show strong convergence for the density in $L^p(0,\tF;L^p(\Omega))$ for all $p\in\lbrack 1,\infty\rparen$.
  Finally, in \textbf{Step 4} we show that the limit solves the momentum equation \eqref{eq:weak:momentum}.
  Notice that all the convergence mentioned above and throughout the rest of the proof are up to the extraction of a subsequence.
  \smallskip

  \noindent \textbf{Step 1}. \textit{Existence of a limit.} %
  The maximum principle \eqref{eq:maximum.principle} yields the existence of $\rho\in L^\infty(0,\tF;L^{\infty}(\Omega))$ such that
  \begin{equation}\label{eq:weak.star.conv.density}
    \text{$\rho_h\overset{\ast}{\rightharpoonup}\rho$ weakly-$\ast$ in $L^\infty(0,\tF;L^{\infty}(\Omega))$}.
  \end{equation}
  It only remains to show that the maximum principle \eqref{eq:maximum.principle} also holds for $\rho$.
  By using the definition of weak-$\ast$ convergence, for any Borel set $A\subset\Omega\times(0,\tF)$, denoting by $\chi_A$ the characteristic function of $A$, we have
  \[
  \int_0^{\tF}\int_\Omega (\rho - \lwr{\rho}) \chi_A
  = \lim_{h\to 0} \int_0^{\tF}\int_\Omega (\rho_h - \lwr{\rho}) \chi_A
  \geq 0
  \]
  and
  \[
  \int_0^{\tF}\int_\Omega (\upr{\rho} - \rho) \chi_A
  = \lim_{h\to 0} \int_0^{\tF}\int_\Omega (\upr{\rho} - \rho_h) \chi_A
  \geq 0.
  \]
  This proves that
  \begin{equation}\label{eq:maximum-principle-cont}
    \lwr{\rho}\leq \rho\leq \upr{\rho}.
  \end{equation}

  Now, we make use of \cite[Proposition C.5]{Droniou.Eymard.ea:18} to show that $(\bvec{u}_h)_{h\in\mathcal{H}}$ converges strongly in $L^1(0,\tF;\bvec{L}^2(\Omega))$ to some $\bvec{u}$. For this, we need to check the following three conditions (all the bounds are uniform in $h$):

  \begin{enumerate}
  \item The sequence $(\bvec{u}_h)_{h\in\mathcal{H}}$ is bounded in $L^1(0,\tF;\bvec{L}^2(\Omega))$;

  \item Denoting by $(\bvec{X}_h,\triplenorm{h}{{\cdot}})_{h\in\mathcal{H}}$ a family of normed spaces compactly embedded in $\bvec{L}^2(\Omega)$ in the sense of \cite[Definition C.4]{Droniou.Eymard.ea:18}, the sequence $\left(\norm{L^1(0,\tF;\bvec{X}_h)}{\bvec{u}_h}\right)_{h\in\mathcal H}$ is bounded;

  \item There exists $\eta : (0,\tF)\rightarrow [0,\infty)$ such that $\lim_{\delta\to 0^+} \eta(\delta) = 0$ and, for all $\delta\in (0,\tF)$ and $h\in\mathcal{H}$, it holds,
    \[
    \int_{ \delta}^{\tF} \norm{\bvec{L}^2(\Omega)}{\bvec{u}_h(t) - \bvec{u}_h(t - \delta)}\,dt
    \leq \eta(\delta).
    \]
  \end{enumerate}
  We choose $\bvec{X}_h = \Poly{0}(\Th)^d$, and for any $\bvec{w}_h\in\bvec{X}_h$, we define the following norm:
  \begin{equation}\label{eq:triple.norm}
    \triplenorm{h}{\bvec{w}_h}\coloneqq \min \left\{
    \norm{a,h}{\uvec{v}_h} \st \uvec{v}_h\in\UhZ \text{ such that } \bvec{v}_h = \bvec{w}_h
    \right\}.
  \end{equation}
  The fact that this is a norm follows from the discrete Poincaré inequality~\eqref{eq:discrete.sobolev} with $p = 2$.
  Since both $\bvec{X}_h$ and $\UhZ$ are finite-dimensional spaces, this minimum is attained, i.e. there exists $\uvec{v}_h^*\in\UhZ$ such that $\triplenorm{h}{\bvec{w}_h} = \norm{a,h}{\uvec{v}_h^*}$ with $\bvec{w}_h = \bvec{v}_h^*$.
  The discrete compactness theorem for the HHO method \cite[Theorem 6.41]{Di-Pietro.Droniou:20} along with the norm equivalence in~\eqref{eq:norm.a.h} then show that $\left(\bvec{X}_h,\triplenorm{h}{{\cdot}}\right)_{h\in\mathcal H}$ is compactly embedded in $\bvec{L}^2(\Omega)$. The first two conditions above are then a direct consequence of the a priori bound~\eqref{eq:a.priori.velocity} and of the choice~\eqref{eq:triple.norm} for the triple norm.
  The third condition is proved in Lemma \ref{lem:time.translations}, with $\eta(\delta) \simeq \delta^{\frac12}$. Therefore,
  \begin{equation}\label{eq:uh.strong.L1}
    \bvec{u}_h\rightarrow \bvec{u} \:\:\text{strongly in}\:\: L^1(0,\tF;\bvec{L}^2(\Omega)).
  \end{equation}

  Next we prove that the convergence actually holds in $L^p(0,\tF;\bvec{L}^2(\Omega))$ for all $p\in\lbrack 1,\infty\rparen$.
  By using a classical interpolation result on Lebesgue spaces followed by a triangle inequality, we have
  \[
  \norm{L^p(0,\tF;\bvec{L}^2(\Omega))}{\bvec{u} - \bvec{u}_h}
  \le
  \norm{L^1(0,\tF;\bvec{L}^2(\Omega))}{\bvec{u} - \bvec{u}_h}^{\frac{1}{p}}
  \left(
  \norm{L^{\infty}(0,\tF;\bvec{L}^2(\Omega))}{\bvec{u}}
  + \norm{L^{\infty}(0,\tF;\bvec{L}^2(\Omega))}{\bvec{u}_h}
  \right)^{1-\frac{1}{p}}.
  \]
  The convergence \eqref{eq:uh.strong.L1} and the a priori bound \eqref{eq:a.priori.velocity} (together with the fact that $\norm{\bvec{L}^2(\Omega)}{\bvec{u}_h} \le \norm{0,h}{\uvec{u}_h}$ by \eqref{eq:L2.discrete.norm}) thus show that
  \begin{equation}\label{eq:strong.conv.velocity.lp.l2}
    \text{
      $\bvec{u}_h\rightarrow \bvec{u}$ strongly in $L^p(0,\tF;\bvec{L}^2(\Omega))$ for all  $p\in[1,\infty)$.
      }
  \end{equation}
  Finally, from Proposition~\ref{prop:properties.Gh} and the a priori estimates~\eqref{eq:a.priori.velocity}, we infer that $\bvec{u}\in L^2(0,\tF;\bvec{H}^1_0(\Omega))$
  \begin{gather}\label{eq:weak.convergence.l2.lq.velocity}
    \text{$\bvec{u}_h \rightharpoonup \bvec{u}$ weakly in $ L^2(0,\tF;\bvec{L}^q(\Omega))$
      for all $1 \le q < q^*$,
    }
    \\ \label{eq:weak.convergence.discrete.grad}
    \text{$\Gh \uvec{u}_h \rightharpoonup \GRAD \bvec{u}$ weakly in $L^2(0,\tF;\bbvec{L}^2(\Omega))$}.
  \end{gather}
  Moreover, by the previous convergence results and \eqref{eq:Zh}, for any $q \in L^2(\Omega)$ we have
  $0 = \int_\Omega \Dh \uvec{u}_h \, q = \int_{\Omega} \Gh \uvec{u}_h \colon q \bbvec{I}_d \to \int_{\Omega} \GRAD \bvec{u} \colon q \bbvec{I}_d = \int_\Omega (\DIV \bvec{u}) \, q$, showing that $\bvec{u}$ is divergence-free.
  \smallskip

  \noindent\textbf{Step 2.} \textit{Identification of the limit for the density advection equation.}
  We next prove that $(\rho,\bvec{u})$ is a solution to \eqref{eq:weak:density}. Let $\varphi\in \Cc(\lbrack 0,\tF\rparen\times \overline{\Omega})$.
  The relation \eqref{eq:weak:density} will follow from the convergence
  \[
  \int_0^{\tF} \int_\Omega \big( -\rho_h \partial_t(\pi_h\varphi)
  + d_h(\uvec{u}_h,\rho_h,\pi_h\varphi) \big)
  - \int_\Omega \rho^0\pi_h\varphi(0)
  \rightarrow -\int_0^{\tF} \int_\Omega \big(\rho\partial_t\varphi
  + \rho\bvec{u}\cdot\GRAD\varphi\big)
  -\int_\Omega \rho^0\varphi(0),
  \]
  since the quantity on the left vanishes by \eqref{eq:discrete:density}, as can be checked integrating by parts the temporal derivative in this equation, %
  recalling the initial condition~\eqref{eq:discrete:ic}, %
  and removing $\pi_h$ in front of $\rho_0$ by using its definition.
  By smoothness of $\varphi$, we easily see that $\partial_t (\pi_h\varphi)\to \partial_t\varphi$ in $L^\infty((0,\tF)\times\Omega)$ and $\pi_h\varphi(0)\to\varphi(0)$ in $L^\infty(\Omega)$, so, by~\eqref{eq:weak.star.conv.density}, the first and third terms converge to their counterparts in the limit. It only remains to show that
  \begin{equation}\label{eq:convergence.dh}
    \int_0^{\tF} d_h(\uvec{u}_h,\rho_h,\pi_h\varphi)\rightarrow
    -\int_0^{\tF} \int_\Omega \rho \bvec{u} \cdot \GRAD\varphi.
  \end{equation}
  This convergence is a consequence of Lemma~\ref{lem:dh:consistency}.
  Indeed, integrating in time \eqref{eq:dh:consistency} we have
  \begin{equation}\label{eq:int.dh}
    \int_0^{\tF} d_h(\uvec{u}_h,\rho_h,\pi_h\varphi)=
    -\int_0^{\tF} \int_\Omega \rho_h \bvec{u}_h \cdot \GRAD\varphi
    + \int_0^{\tF}\res
    \eqcolon \term_1 + \term_2.
  \end{equation}
  On one hand, from the weak convergence \eqref{eq:weak.star.conv.density} of $\rho_h$, the strong convergence \eqref{eq:strong.conv.velocity.lp.l2} of $\bvec{u}_h$, and the fact that $\varphi\in \Cc(\lbrack 0,\tF \rparen \times \Omega)$, we have $\term_1\rightarrow -\int_0^{\tF} \int_\Omega \rho \bvec{u}\cdot\GRAD\varphi$. On the other hand, the estimate \eqref{eq:dh:consistency:residual} and Cauchy--Schwarz inequalities show that
  \begin{align*}
    |\term_2|\lesssim{}& h^{\frac12}\left(\int_0^{\tF}\seminorm{\bvec{u},\upw,h}{\rho_h(s)}^2\,ds\right)^{\frac12}\left(\int_0^{\tF}\norm{a,h}{\uvec{u}_h(s)}\,ds\right)^{\frac12}\norm{L^\infty(0,\tF;W^{1,4}(\Omega))}{\varphi}\\
    &+h\upr{\rho}\left(\int_0^{\tF}\norm{a,h}{\uvec{u}_h(s)}\,ds\right)\norm{L^\infty(0,\tF;H^{1}(\Omega))}{\varphi}
    \lesssim h^{\frac12}+h,
  \end{align*}
  the conclusion following from the estimates in Lemma \ref{lem:a.priori}. This shows that $\term_2\rightarrow 0$ and thus that the convergence \eqref{eq:convergence.dh} holds.
  \smallskip

  \noindent \textbf{Step 3.} \textit{Strong convergence for the density.}
  It is sufficient to show that $\rho_h\rightarrow\rho$ strongly in $L^2(0,\tF;L^2(\Omega))$. Since $\rho_h\in L^{\infty}(0,\tF;L^{\infty}(\Omega))$, by interpolation we will deduce the strong convergence in $L^p(0,\tF;L^p(\Omega))$ for any $p\in[1,\infty)$.

    From \eqref{eq:weak.star.conv.density}, we get that $\rho_h\rightharpoonup\rho$ weakly in $L^2(0,\tF;L^2(\Omega))$, leading to
    \begin{equation}\label{eq:lim.inf.discrete.density}
      \norm{L^2(0,\tF;L^2(\Omega))}{\rho}\leq\liminf_{h\to 0}\norm{L^2(0,\tF;L^2(\Omega))}{\rho_h}.
    \end{equation}
    In addition, from the a priori estimate \eqref{eq:a.priori.density} we deduce
    \begin{equation}\label{eq:discrete.density.bound}
      \norm{L^2(\Omega)}{\rho_h(t)}^2\leq\norm{L^2(\Omega)}{\rho^0}^2\quad\forall t\in [0,\tF].
    \end{equation}
    The theory of renormalised solutions from \cite{DiPerna.Lions:89} allows us to use a fundamental result \cite[Theorem VI.1.3]{Boyer.Fabrie:13}: if $\bvec{u} \in L^1(0,\tF; \bvec{W}^{1,1}(\Omega))$ and $\rho \in L^{\infty}((0,\tF) \times \Omega)$, then, for any differentiable function $\beta \colon \Real \to \Real$, $\beta(\rho)$ is a weak solution of \eqref{eq:density}, and thus satisfies \eqref{eq:weak:density} for all $\varphi\in C^\infty_c([0,\tF]\times \overline{\Omega})$.
    Given the continuous embedding $\bvec{u}\in L^2(0,\tF;\bvec{Z}) \hookrightarrow L^1(0,\tF; \bvec{W}^{1,1}(\Omega))$, we can apply this result to our case and select $\beta : \Real \ni \xi  \mapsto \xi^2 \in \Real$ and $\varphi(\cdot, t) = \tF-t$ for all $t \in [0,\tF]$.
    Plugged into the equivalent of  \eqref{eq:weak:density} for $\rho^2$, this gives
    \begin{equation*}
      \norm{L^2(0,\tF;L^2(\Omega))}{\rho}^2 =\tF\norm{L^2(\Omega)}{\rho^0}^2.
    \end{equation*}
    Therefore, from \eqref{eq:discrete.density.bound}, and integrating in time, we have
    \begin{equation}\label{eq:lim.sup.discrete.density}
      \limsup_{h\to 0} \norm{L^2(0,\tF;L^2(\Omega))}{\rho_h} \leq \norm{L^2(0,\tF;L^2(\Omega))}{\rho},
    \end{equation}
    which, combined with \eqref{eq:lim.inf.discrete.density}, proves that
    $\rho_h\to\rho$ strongly in $L^2(0,\tF;L^2(\Omega))$ as $h\to 0$.

    \noindent\textbf{Step 4.} \textit{Identification of the limit for the momentum balance equation.} We next prove that the limit $(\rho,\bvec{u})$ satisfies the momentum balance equation \eqref{eq:weak:momentum}.
    For $\bvec{v}\in \Cc(\lbrack 0,\tF\rparen \times \Omega)^d$ with $\DIV\bvec{v} = 0$, setting $\huvec{v}_h \coloneqq \Ih \bvec{v}$, our goal is to pass to the limit as $h \to 0$ in~\eqref{eq:discrete:momentum.bis} integrated over $(0,\tF)$:
    \begin{multline*}
      \int_{0}^{\tF} \left[
        \int_{\Omega}\partial_t (\rho_h \bvec{u}_h)\cdot\hbvec{v}_h
        + \lwr{\rho} j_h(\partial_t\uvec{u}_h,\huvec{v}_h)
        \right]
      + \int_{0}^{\tF} \mu a_h(\uvec{u}_h, \huvec{v}_h)
      \\
      + \int_{0}^{\tF} \left[
        c_h(\underline{(\rho\bvec{u})}_h, \uvec{u}_h, \huvec{v}_h)
        + \frac12 d_h(\uvec{u}_h, \rho_h(t), \bvec{u}_h\cdot\hbvec{v}_h)
        \right]
      = \int_{0}^{\tF} \int_\Omega \bvec{f} \cdot \hbvec{v}_h.
    \end{multline*}
    Using the strong convergence of $(\hbvec{v}_h)_{h \in \mathcal{H}}$ to $\bvec{v}$ in $L^2(0,\tF; \bvec{L}^2(\Omega))$, we readily infer
    \[
    \int_0^{\tF} \int_\Omega \bvec{f} \cdot \hbvec{v}_h \to \int_0^{\tF} \int_\Omega \bvec{f} \cdot \bvec{v}.
    \]
    Denote by $\term_1,\ldots,\term_3$ the terms in the left-hand side.
    We start by showing that
    \begin{equation}\label{eq:Step.3:T1:convergence}
      \term_1
      \to -\int_{0}^{\tF} \int_{\Omega} \rho \bvec{u} \cdot \partial_t \bvec{v}
      + \int_{\Omega} \rho^0 \bvec{u}^0 \bvec{v}(0).
    \end{equation}
    Integrating by parts in time and recalling the initial condition~\eqref{eq:discrete:ic}, we get
    \begin{equation}\label{eq:Step.3:T1:convergence:basic}
      \term_1
      = -\int_{0}^{\tF}\int_{\Omega} \rho_h \bvec{u}_h\cdot \partial_t \hbvec{v}_h
      + \int_{\Omega} \pi_h\rho^0 \pi_h\bvec{u}^0\cdot \hbvec{v}_h(0)
      - \lwr{\rho}\int_{0}^{\tF}j_h(\uvec{u}_h,\partial_t\huvec{v}_h)
      + \lwr{\rho} j_h(\Ih\bvec{u}^0,\huvec{v}_h(0)).
    \end{equation}
    The two rightmost (stabilisation) terms vanish as $h\to 0$ by using the Cauchy--Schwarz inequality, the property \eqref{eq:jh.assumption} of $j_h$ for interpolates of smooth functions, and the bound on $j_h(\uvec{u}_h,\uvec{u}_h)\le \norm{0,h}{\uvec{u}_h}^2$ in \eqref{eq:a.priori.velocity}, so we only need to check the remaining terms.
    By strong convergence of the projections, and since $\hbvec{v}_h(0) \to \bvec{v}(0)$ strongly in $\bvec{L}^2(\Omega)$, the convergence of the second integral in the right-hand side is guaranted.
    On the other hand, $\partial_t \hbvec{v}_h \to \partial_t \bvec{v}$ strongly in $L^2(0,\tF; \bvec{L}^2(\Omega))$ and $\bvec{u}_h \to \bvec{u}$ strongly in $L^2(0,\tF; \bvec{L}^2(\Omega))$, which gives $\bvec{u}_h \cdot \partial_t \hbvec{v}_h \to \bvec{u} \cdot \partial_t \bvec{v}$ strongly in $L^1(0,\tF;L^1(\Omega))$; the fact that $\rho_h \overset{*}\rightharpoonup \rho$ weakly-* in $L^{\infty}(0,\tF;L^{\infty}(\Omega))$ then shows that
    $
    \int_{0}^{\tF}\int_{\Omega}\rho_h \bvec{u}_h \cdot \partial_t \hbvec{v}_h \to \int_{0}^{\tF}\int_{\Omega}\rho \bvec{u} \cdot \partial_t \bvec{v},
    $
    which justifies the convergence of the first integral in the right-hand side of~\eqref{eq:Step.3:T1:convergence:basic} and concludes the proof of~\eqref{eq:Step.3:T1:convergence}.

    The convergence of $\term_2$ towards $\int_0^{\tF} \int_\Omega \mu \GRAD \bvec{u} : \GRAD \bvec{v}$ is a simple consequence of the sequential compactness \eqref{eq:sequential.compactness.viscous.term} since $(\norm{L^2(0,\tF;\UhZ)}{\uvec{u}_h})_{h\in\mathcal H}$ is bounded by \eqref{eq:a.priori.velocity}.

    The convergence of $\term_3$ is a consequence of Lemma~\ref{lem:ch:consistency}, as detailed hereafter.
    Let us first show that, for the term resulting from the integral over $(0,\tF)$ of the first addend in the right-hand side of~\eqref{eq:ch:consistency}, we have
    \[
    -\int_{0}^{\tF} \int_{\Omega} (\rho_h \bvec{u}_h \otimes \bvec{u}_h) : \Gh \huvec{v}_h
    \to -\int_{0}^{\tF} \int_{\Omega} (\rho \bvec{u} \otimes \bvec{u}) : \GRAD \bvec{v}.
    \]
    By Step 1, $\bvec{u}_h \to \bvec{u}$ strongly in $L^2(0,\tF; \bvec{L}^2(\Omega))$, and thus $\bvec{u}_h \otimes \bvec{u}_h \to \bvec{u} \otimes \bvec{u}$ strongly in $L^1(0,\tF; \bvec{L}^1(\Omega))$,
    Besides, $\rho_h \overset{*}\rightharpoonup \rho$ weakly-* in $L^{\infty}(0,\tF;L^{\infty}(\Omega))$, hence, $\rho_h \bvec{u}_h \otimes \bvec{u}_h \rightharpoonup \rho \bvec{u} \otimes \bvec{u}$ weakly in $L^{1}(0,\tF; \bbvec{L}^1(\Omega))$.
    Finally, by the consistency of the gradient reconstruction for smooth functions \cite[Proposition 9.14, Point (ii)]{Di-Pietro.Droniou:20}, we have that $\Gh \huvec{v}_h \to \GRAD \bvec{v}$ strongly in $L^{\infty}(0,\tF; \bbvec{L}^{\infty}(\Omega))$, which shows that $(\rho_h \bvec{u}_h \otimes \bvec{u}_h) : \Gh \huvec{v}_h \rightharpoonup (\rho \bvec{u} \otimes \bvec{u}) : \GRAD \bvec{v}$ weakly in $L^1(0,\tF;L^1(\Omega))$.
    To conclude, it suffices to observe that, by~\eqref{eq:ch:consistency:residual}, the residual term in~\eqref{eq:ch:consistency} satisfies
    \begin{align*}
      \left|\int_0^{\tF}\res\right|\lesssim{}& 
      h \upr{\rho}\norm{L^2(0,\tF;\UhZ)}{\uvec{u}_h}^2 \norm{L^\infty(0,\tF;\bvec{W}^{1,4}(\Omega))}{\bvec{v}}
      \\
      &+ h^{\frac12}\left(\int_0^{\tF}\seminorm{\bvec{u},\upw,h}{\rho_h(s)}^2\,ds\right)\left(\max_{t\in[0,\tF]}\norm{0,h}{\uvec{u}_h(t)}\right)^{\frac12}\norm{L^2(0,\tF;\UhZ)}{\uvec{u}_h} \norm{L^\infty(0,\tF;\bvec{W}^{1,\infty}(\Omega))}{\bvec{v}},
    \end{align*}
    and that the right-hand side above tends to $0$ as $h\to 0$ thanks to the estimates \eqref{eq:a.priori}.
\end{proof}

\section*{Acknowledgements}

Funded by the European Union (ERC Synergy, NEMESIS, project number 101115663).
Views and opinions expressed are however those of the authors only and do not necessarily reflect those of the European Union or the European Research Council Executive Agency. Neither the European Union nor the granting authority can be held responsible for them.

%------------------------------------------------------------------------------%

\printbibliography

\end{document}